\documentclass[11pt]{amsart}
\usepackage[english]{babel} 
\usepackage[latin1]{inputenc} 
\usepackage{amssymb,geometry,color,graphicx}
\usepackage{amsmath,amsthm} 
\usepackage{booktabs}
\usepackage{datetime}
\usepackage{dsfont}  
\usepackage{amstext} 
\usepackage{graphicx}   
\usepackage{fancyhdr} 
\usepackage{latexsym}
\usepackage{mathrsfs}
\usepackage{cancel}
\usepackage{cases}
\usepackage{mathtools}
\usepackage{bm}
\usepackage{caption}
\usepackage{subcaption}
\usepackage{tikz}
\usepackage{hyperref}
\usepackage{enumitem}
\usepackage{multirow}
\usepackage{booktabs}
\usepackage{amsmath}
\usepackage{listings} 

\hypersetup{
    colorlinks=true, 
    linktoc=all,     
    linkcolor=blue,  
}

\usepackage{changes}
\definechangesauthor[name=Pawel, color=blue]{PP}
\definechangesauthor[name=Fabio, color=teal]{FVD}

\numberwithin{equation}{section} 

\newcommand{\udef}{\mathrel{\mathop:}=}

\newcommand{\R}{\mathbb{R}}

\newcommand{\N}{\mathbb{N}}

\newcommand{\de}{\mathrm{d}}

\theoremstyle{plain}
\newtheorem{thm}{Theorem}[section]

\newtheorem{lem}[thm]{Lemma}

\newtheorem{rmk}[thm]{Remark}

\newtheorem{rem}[thm]{Remark}

\usepackage{listings} 	
\usepackage{cancel}
\title{A Randomized Runge-Kutta Method for time-irregular delay differential equations}

\author[F.V. Difonzo]{Fabio~V. Difonzo}
\address{Istituto per le Applicazioni del Calcolo \textquotedblleft Mauro Picone\textquotedblright, Consiglio Nazionale delle Ricerche, Via G. Amendola 122/I, 70126 Bari, Italy}
\email{fabiovito.difonzo@cnr.it}

\author[P. Przyby{\l}owicz]{Pawe{\l} Przyby{\l}owicz}
\address{AGH University of Krakow,
Faculty of Applied Mathematics,
Al. A.~Mickiewicza 30, 30-059 Krak\'ow, Poland}
\email{pprzybyl@agh.edu.pl}

\author[Y. Wu]{Yue Wu}
\address{Department of Mathematics and Statistics, University of Strathclyde, Glasgow, UK}
\email{yue.wu@strath.ac.uk, corresponding author}

\author[X. Xie]{Xinheng Xie}
\address{Department of Mathematics and Statistics, University of Strathclyde, Glasgow, UK}
\email{xinheng.xie@strath.ac.uk}

\begin{document}
\maketitle

\begin{abstract}
{In this paper we investigate the existence, uniqueness and approximation of solutions of delay differential equations (DDEs) with the right-hand side functions $f=f(t,x,z)$ that are Lipschitz continuous with respect to $x$ but only H\"older continuous with respect to $(t,z)$. We give a construction of the  randomized two-stage Runge-Kutta scheme for DDEs and investigate its upper error bound in the $L^p(\Omega)$-norm for $p\in [2,+\infty)$. Finally, we report on results of numerical experiments.}
\end{abstract}

\section{Introduction}

We deal with approximation  of solutions to the following delay differential equations (DDEs)
{
\begin{equation}
\label{eq:DiscDDE2}
\begin{cases}
x'(t)=f(t, x(t),x(t-\tau)), & t\in [0,(n+1)\tau], \\
x(t)=\varphi(t), & t\in[-\tau,0),
\end{cases}
\end{equation}}
with a constant time-lag $\tau\in(0,+\infty)$, a fixed time horizon $n\in\mathbb{N}$, a right-hand side function  $f:[0,(n+1)\tau]\times\R^d\times\R^d\mapsto\R^d$, and {initial-value function: $\varphi(t): [-\tau,0)\mapsto \R^d$}. 

We assume that the function $f=f(t,x,z)$ is integrable with respect to $t$ and (at least) continuous with respect to $(x,z)$.

{Building upon the concepts presented in \cite{difonzo2024existence}, our objective in this paper is to introduce a Randomized Runge-Kutta scheme tailored specifically for time-irregular delay differential equations. This novel scheme differentiates itself from existing methods, including those analyzed in the literature. 

The motivation behind exploring these Delay Differential Equations (DDEs) is multifaceted. One aspect stems from the need to model switching systems with memory, as expounded upon in \cite{hale1977theory} and \cite{Hale1993IntroductionTF}. Another facet is rooted in practical engineering applications, as evidenced by discussions in \cite{CZPMPP} and \cite{Eng_DDEs_1}. For instance, in scenarios like the phase change of metallic materials, a DDE becomes crucial due to the time delay in response to changes in processing conditions. 

Additionally, inspiration is drawn from delayed differential equations involving rough paths, as represented by the form:
\begin{equation}\label{eq:DiscDDE21}
\begin{cases}
\de U(t)=a(U(t),U(t-\tau))+\de Z(t), & t\in [0,(n+1)\tau], \\
U(t)=U_{0}, & t\in[-\tau,0),
\end{cases}
\end{equation}

Here, $Z$ symbolizes an integrable perturbation, potentially of stochastic nature, with paths that might exhibit discontinuities. Consequently, if $x(t)=U(t)-Z(t)$, it satisfies the (possibly random) DDE \eqref{eq:DiscDDE2} with $f(t,x,z)=a(x+Z(t),z+Z(t-\tau))$, and $x(t)=U_0$, assuming $Z(t)=0$ for $t\in [-\tau,0]$. In this context, the function $f$ inherits from $Z$ its low smoothness concerning the variable $t$. Noteworthy is the fact that equation \eqref{eq:DiscDDE21} is a generalization of an ODE with rough paths discussed in \cite{RKYW2017}.

When classical assumptions, such as $C^r$-regularity of $f=f(t,x,z)$ concerning all variables $t,x,z$, are imposed on the right-hand side function, errors for deterministic schemes have been documented in \cite{bellen1}. Furthermore, in \cite{CZPMPP}, the error of the Euler scheme has been explored for a certain class of nonlinear DDEs under nonstandard assumptions like the one-side Lipschitz condition and local H\"older continuity. However, limited knowledge exists regarding the approximation of solutions for DDEs with less regular Carath\'eodory right-hand side functions.

For Carath\'eodory ODEs, deterministic algorithms encounter convergence issues, necessitating the use of randomized algorithms. In this context, we propose a randomized version of the Runge-Kutta scheme tailored for DDEs of the form \eqref{eq:DiscDDE2}.\\
In the landscape of existing literature, several relevant contributions are worth noting. \cite{Eulalia2001} provides an analysis of multistep Runge-Kutta methods for delay differential equations, while \cite{Kuehn2004} offers comprehensive insights into numerical methods for delay differential equations. The work by \cite{Brunner2012} introduces a general-purpose implicit-explicit Runge-Kutta integrator for delay differential equations, and \cite{Song2015} proposes a novel explicit two-stage Runge-Kutta method for delay differential equations with constant delay. Each of these works contributes valuable perspectives to the broader understanding of numerical methods for DDEs. 

Despite the extensive exploration of randomized algorithms for ODEs in the literature (see, for instance, \cite{bochacik2}, \cite{BGMP2021}, \cite{daun1}, \cite{hein_milla1}, \cite{jen_neuen1}, \cite{BK2006}, \cite{RKYW2017}), to the best of our knowledge, this paper represents a pioneering effort in defining a randomized Runge-Kutta scheme and rigorously investigating its error for Carath\'{e}odory-type DDEs.
}

{The structure of the paper is as follows. In section 2 we give basic notions, definitions, and provide detailed construction and output of the randomized two-stage Runge-Kutta method. Section 3 is devoted to properties of solutions to CDDEs under assumptions stated in Section 2. Section 4 contains detailed error analysis of the randomized R-K method for CDDEs. Finally, details of the Python implementation  together with extensive numerical experiments are reported in Section 5.}

\section{Preliminaries}

 By $|\cdot|$ we mean the Euclidean norm in $\R^d$. We consider a  complete probability space $(\Omega,\Sigma,\mathbb{P})$. For a random variable $X:\Omega\to\mathbb{R}$ we denote by $\|X\|_p=(\mathbb{E}|X|^p)^{1/p}$, where $p\in [2,+\infty)$.

Let us fix the {\it horizon parameter} $n\in\mathbb{N}$. On the right-hand side function $f$ { and initial-value function: $\varphi(t)$,} we impose the following assumptions:


\begin{enumerate}[label=\textbf{(A\arabic*)},ref=(A\arabic*)]
       \item\label{ass:A1} $f(t,\cdot,\cdot)\in C(\R^d\times\R^d;\R^d)$ for all $t\in [0,(n+1)\tau]$ { and
     $\varphi\in C([-\tau,0);\mathbb{R}^d)$ for $t\in [-\tau, 0]$.}
    \item\label{ass:A2} $f(\cdot,x,z)$ is Borel measurable for all $(x,z)\in \R^d\times\R^d$,
    \item\label{ass:A3}
    {There exists $K\in (0,\infty)$ such that for all $t\in [0,(n+1)\tau]$, $x, z\in\R^d$}
    \begin{equation}
     \label{assumpt_a3}
       { |f(t,x,z)|\leq K (1+|x|)(1+|z|),}
    \end{equation}
    
    \item\label{ass:A4}
    There exists $L\in (0,\infty)$ such that for all $t\in [0,(n+1)\tau]$, $x_1,x_2, z\in\R^d$     
    \begin{equation}
    \label{assumpt_a4}
        {|f(t,x_1,z)-f(t,x_2,z)|\leq L (1+|z|)|x_1-x_2|,}
    \end{equation}
\end{enumerate}

In Section 3,  under the assumptions \ref{ass:A1}-\ref{ass:A4}, we investigate existence and uniqueness of solution for \eqref{eq:DiscDDE2}. Next, in Section 4 we investigate error of the { {\it randomized two-stage Runge-Kutta scheme} under slightly stronger assumptions.  Namely, we impose H\"older continuity assumptions on $f=(t,x,z)$ with respect to $(t,z)$.}

{The definition of the two-stage randomized Runge-Kutta method for DDEs goes as follows.} $(\gamma_k^j)_{j\in\mathbb{N}_0,k\in\mathbb{N}}$ iid from $\mathcal{U}(0,1)$, $N\in\mathbb{N}$, $h=\tau/N$, $t_k^j=j\tau+kh$, $\theta_{k+1}^j=t_k^j+h\gamma_{k+1}^j$ for $k=0,1,\ldots,N$, {$j=-1,0,1,\ldots,n$. Also let $h_{k+1}^j:=h\gamma_{k+1}^j$
and define $y_k^{-1}=\varphi(t_k^{-1})$.

For $j= 0$, define
\begin{align}\label{eqn:rrk_0}
    \begin{split}
       &y_0^0=y_N^{-1},\\
    &\tilde y_{k+1}^{-1,0}=\varphi(t_k^{-1}+h_{k+1}^0)\\
    &\tilde y_{k+1}^0=y_k^0+h_{k+1}^0\cdot f(t_k^0,y_k^0,y_k^{-1}),\\
    &y_{k+1}^0=y_k^0+h\cdot f(\theta_{k+1}^0,\tilde y_{k+1}^0,\tilde y_{k+1}^{-1,0}),   
    \end{split}
\end{align}

and for $j\geq 1$,
\begin{align}\label{eqn:rrk_j}
    \begin{split}
            &y_0^j=y_N^{j-1},\\
    &\tilde y_{k+1}^{j-1,j}=y_k^{j-1}+h_{k+1}^j\cdot f(t_k^{j-1},y_k^{j-1},y_{k}^{j-2}),\\
    &\tilde y_{k+1}^j=y_k^j+h_{k+1}^j\cdot f(t_k^j,y_k^j,y_k^{j-1}),\\
    &y_{k+1}^j=y_k^j+h\cdot f(\theta_{k+1}^j,\tilde  y_{k+1}^j,\tilde  y_{k+1}^{j-1,j}).
    \end{split}
\end{align}

Note that for $j=1$, the delay term $y^{j-2}_k$ in the second line of \eqref{eqn:rrk_j} is exactly the evaluation of the initial condition $\varphi (t^{-1}_k)$.

Let us call $\tilde y_{k+1}^{j-1,j}$ and $\tilde y_{k+1}^{j}$ \textit{the intermediate delay term} and \textit{the intermediate term} respectively. A key component here is the intermediate delay term $\tilde y_{k+1}^{j-1,j}$: we do not recycle the intermediate term $\tilde y_{k+1}^{j-1}$ computed from preceding $\tau$-interval $[(j-1)\tau, j\tau]$ to replace $\tilde y_{k+1}^{j-1,j}$ for a simplified computation, because $\tilde y_{k+1}^{j-1}$ is generated on $\gamma^{j-1}_{k+1}$, a different random resource from $\gamma^{j}_{k+1}$. {It is easy to see that each random vector $y_k^j$, $j=0,1,\ldots,n$, $k=0,1,\ldots,N$, is measurable with respect to the $\sigma$-field generated by the following family of independent random variables
\begin{equation}
\label{sig_field1}
\Bigl\{\theta_1^0\ldots,\theta_N^0,\ldots,\theta_1^{j-1},\ldots,\theta_N^{j-1},\theta_1^j,\ldots,\theta_k^j\Bigr\}.
\end{equation}

As the horizon parameter $n$ is fixed, the randomized two-stage Runge-Kutta scheme uses $O(N)$ evaluations of $f$ (with a constant
in the $O$ notation that only depends on $n$ but not on $N$).}
}

{To help readers better understand the randomized RK method, we present two figures \ref{fig:RRKj0} and \ref{fig:RRKj1}. Figure \ref{fig:RRKj0} depicts the scenario when $j=0$, while Figure \ref{fig:RRKj1} showcases the situation for $j\ge 1$, using $j=1$ as a specific example. Within these figures, different markers are employed to convey distinct meanings: dashed lines represent assignments; black solid lines and solid circles signify computation methods; colored solid circles indicate the values of different terms at the time $t$; the black dots represent the value of time $t$. For clarity, text is color-coded to match the terms they represent. Structurally, the figures are divided into two primary sections vertically along the coordinate axis. The upper section encompasses all computations excluding the generation of $y_{i+1}^j$, while the lower section specifically illustrates the operations involved in generating $y_{i+1}^j$.
In both figures, for any $i$ in $[0,1,\ldots,N-1]$, we aim to compute the value of $y_{i+1}^j$ at the time point $t_{i+1}^j$. We assume that the values of $y_i^j$ for all time points before $t_{i+1}^j$ are known. Utilizing these known values, we can calculate the value of $y_{i+1}^j$ at the time point $t_{i+1}^j$ using Eqn. \eqref{eqn:rrk_0} if $j = 0$ or Eqn. \eqref{eqn:rrk_j} if $j \ge 1$.}

\begin{figure}[h]
  \centering
  \begin{subfigure}{0.95\textwidth}
    \includegraphics[width=\textwidth]{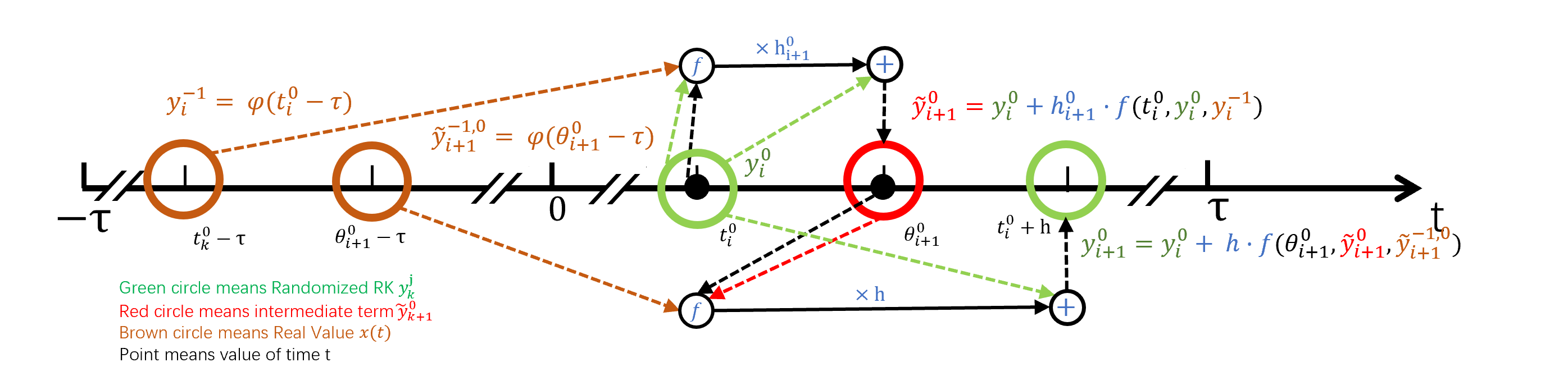} 
    \caption{Scenario for $j=0$}
    \label{fig:RRKj0}
  \end{subfigure}
  
  \vspace{1cm}
  
  \begin{subfigure}{0.95\textwidth}
    \includegraphics[width=\textwidth]{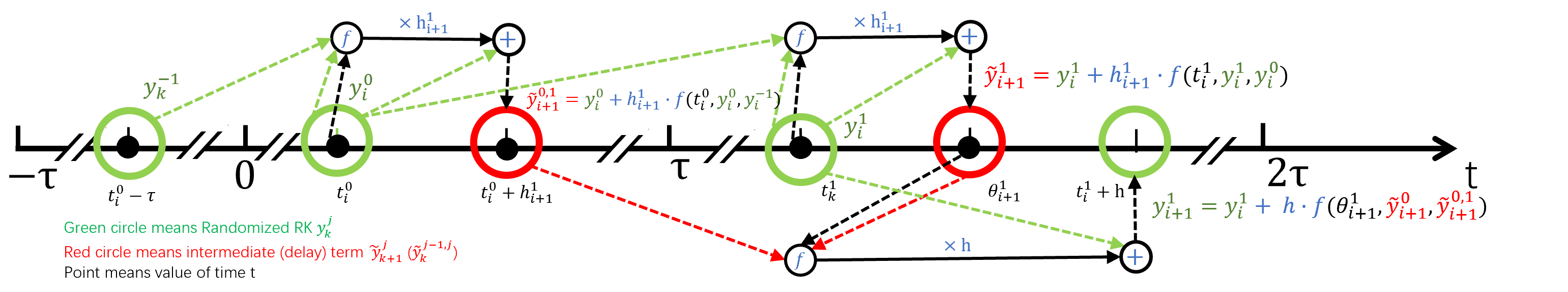} 
    \caption{Scenario for $j>0$ with $j=1$ exemplified}
    \label{fig:RRKj1}
  \end{subfigure}
  
  \caption{Visual representation of Randomized RK method for different $j$.}
  \label{fig:combined}
\end{figure}

\section{Properties of solutions to  Carath\'eodory DDEs}
In this section we investigate the issue of existence and uniqueness of the solution of \eqref{eq:DiscDDE2} under the assumptions \ref{ass:A1}-\ref{ass:A4}. 

In the sequel we use the following equivalent representation of the solution of \eqref{eq:DiscDDE2}, that is very convenient when proving its properties and when estimating the error of the randomized RK scheme. For $j=0,1,\ldots,n$ and $t\in [0,\tau]$ it holds
\begin{equation}
    x'(t+j\tau)=f(t+j\tau,x(t+j\tau),x(t+(j-1)\tau)).
\end{equation}

Hence, we take {$\phi_{-1}(t):=\varphi(t-\tau)$},  $\phi_j(t)\udef x(t+j\tau)$ and for $j=0,1,\ldots,n$ we consider the following sequence of initial-value problems
\begin{equation}
\label{eqODE_j0}
\begin{cases}
\phi_j'(t)=g_j(t,\phi_j(t)), & t\in [0,\tau], \\
\phi_j(0)=\phi_{j-1}(\tau), 
\end{cases}
\end{equation}

with $g_j(t,x)=f(t+j\tau,x,\phi_{j-1}(t))$, $(t,x)\in [0,\tau]\times\mathbb{R}^d$. Then the solution of \eqref{eq:DiscDDE2} can be written as
  \begin{equation}
        x(t)=\sum\limits_{j=-1}^n \phi_j(t-j\tau)\cdot\mathbf{1}_{[j\tau,(j+1)\tau]}(t), \quad t\in [-\tau,(n+1)\tau].
    \end{equation}

We prove the following result about existence, uniqueness and H\"older regularity of the solution of the delay differential equation \eqref{eq:DiscDDE2}. We will use this theorem in the next section when proving error estimate for the randomized RK algorithm. 
\begin{thm}
\label{sol_dde_prop_1}
    Let $n\in\N\cup\{0\}$, $\tau\in (0,+\infty)$, $x_0\in\R^d$ and let $f$, {$\varphi$} satisfy {the} assumptions \ref{ass:A1}-\ref{ass:A4}. Then there exists a unique absolutely continuous solution $x=x(x_0,f)$ to \eqref{eq:DiscDDE2} such that for $j=0,1,\ldots,n$ we have
    \begin{equation}
    \label{upper_est_Kj}
        \sup\limits_{0\leq t \leq \tau}|\phi_j(t)|\leq K_j
    \end{equation}
    
    where  {$K_{-1}:=\max_{t\in [-\tau,0]}|\varphi(t)|$} and
    {
    \begin{equation}
    \label{def_K_j}
        K_j=(1+K_{j-1})(1+K\tau )\cdot\exp\Bigl((1+K_{j-1})K\tau \Bigr).
    \end{equation}
    }
    
    then for all $j=0,1,\ldots,n$, $t,s\in [0,\tau]$ it holds
   {
    \begin{equation}
    \label{phi_j_Lipschitz}
        |\phi_j(t)-\phi_j(s)|\leq (1+K_{j-1})(1+K_{j}) K|t-s|.
    \end{equation}}
\end{thm}

{
The proof of this theorem follows Theorem 3.1 from \cite{difonzo2024existence}. However, due to different initial conditions and assumptions, there are some subtle variations in the proof and conclusions. Specifically, the differences are as follows:

\begin{enumerate}
    \item \cite{difonzo2024existence} considers a constant initial condition case, ie, $\varphi(t)=x_0$ for $t\in [-\tau, 0]$, while in our DDE, a general initial condition $\varphi(t)$ is taken, allowing the initial condition to vary with time.
    
    \item Compared to \cite[Assumptions A3 and A4]{difonzo2024existence}, we impose stronger assumptions \ref{ass:A3} and \ref{ass:A4} on $f$ such that the coefficients of linear growth and Lipschitz condition are time-invariant.
\end{enumerate}
\begin{rmk}
  The change mentioned in (2) above results in a stronger regularity of the auxiliary ODEs \eqref{eqODE_j0}: \cite[Lemma 7.1]{difonzo2024existence} claims that the true solution is H\"older-continuous while, under our assumptions, the exact solution is Lipschitz as shown in \eqref{phi_j_Lipschitz}. 
\end{rmk}
\begin{proof}
Follow the proof of Theorem 3.1 in \cite{difonzo2024existence}, we proceed by induction. We start with the case when $j=0$ and consider the following initial-value problem. 
\begin{equation}
\label{eqODE_0}
    \begin{cases}
    \phi_0'(t)=g_0(t,\phi_0(t)), & t\in [0,\tau], \\
        {\phi_0(0)=\varphi(0)},
    \end{cases}
\end{equation}

with $g_0(t,x)=f(t,x,\phi_{-1}(t))=f(t,x,x_0)$. It is obvious that for all $t\in [0,\tau]$ the function $g_0(t,\cdot)$ is continuous and for all $x\in\R^d$ the function $g_0(\cdot,x)$ is Borel measurable. {Moreover, by \eqref{assumpt_a3} we have 
$${|g_0(t,x)|\leq K(1+|\varphi(t-\tau)|)(1+|x|)\leq K(1+K_{-1})(1+|x|)}$$ for all $(t,x)\in [0,\tau]\times\R^d$, and by \eqref{assumpt_a4} {there exists $L\in (0,\infty)$ such that for all $t\in [0,\tau]$, $x,y\in\R^d$ we have}}
\begin{equation}
    {|g_0(t,x)-g_0(t,y)|\leq L (1+|\varphi(t-\tau)|) |x-y|\leq L(1+K_{-1})|x-y|.}
\end{equation}

Therefore, by Lemma \ref{lem_ode_1} we have that there exists a unique absolutely continuous solution $\phi_0:[0,\tau]\to\R^d$ of \eqref{eqODE_0} that satisfies \eqref{upper_est_Kj} with $j=0$. {In addition, by Lemma \ref{lem_ode_1} we obtain that $\phi_0$ satisfies \eqref{phi_j_Lipschitz} for $j=0$. }
For the inductive step from $j$ to $j+1$, we can simply follow the approach provided \cite{difonzo2024existence}, since our assumptions \ref{ass:A3} and \ref{ass:A4} are stronger than \cite[Assumptions A3 and A4]{difonzo2024existence}. It is crucial to note that \cite[Eqn. (3.14)]{difonzo2024existence} will be modified to 
\[|\phi_{j+1}(t)-\phi_{j+1}(s)| \leq (1+K_{j})(1+K_{j+1}) K |t-s|,\]

due to the differences between our Lemma \ref{lem_ode_1} and \cite[Lemma 7.1]{difonzo2024existence}. This concludes the inductive proof.
\end{proof}
}

\section{Error of the randomized  RK scheme }
In this section we perform detailed  error analysis for the randomized RK scheme. As mentioned in Section 1, for the error analysis we impose global Lipschitz assumption on $f=f(t,x,z)$ with respect to $x$ together with global H\"older condition with respect to $z$. Namely, instead of \ref{ass:A3} and \ref{ass:A4}, we assume
{
\begin{enumerate}[label=\textbf{(A3')},ref=(A3')]
    \item\label{ass:A3'}
     there exist $\bar K \in [0,+\infty)$ such that for all $t\in [0,(n+1)\tau]$, 
    \begin{equation}
        |f(t,0,0)|\leq \bar K,
    \end{equation}
     and there exist $L \in [0,+\infty)$, {$\alpha,\gamma\in (0,1]$} such that:\\ 
     1. for all $t\in [0,(n+1)\tau]$, $x_1,x_2,z_1,z_2\in\R^d$, 
    \begin{equation}\label{eqn:f_assum_xz}
        |f(t,x_1,z_1)-f(t,x_2,z_2)|\leq L\Bigl( |x_1-x_2|+|z_1-z_2|^{\alpha}\Bigr),
    \end{equation}
    2. for all $t_1, t_2\in [0,(n+1)\tau]$, $x,z\in\R^d$, 
    \begin{equation}
        |f(t_1,x,z)-f(t_2,x,z)|\leq L(1+|x|+|z|)|t_1-t_2|^\gamma.
    \end{equation}
    3. for all $s,t\in [-\tau, 0]$, 
    \begin{equation}
    |\varphi(t)-\varphi(s)|\leq L|t-s|, \quad s,t\in [-\tau, 0].
    \end{equation}
\end{enumerate}}

\begin{rem}
\label{rem_str_assum}
Note that the assumptions \ref{ass:A1}, \ref{ass:A2}, \ref{ass:A3'} are stronger than the assumptions \ref{ass:A1}-\ref{ass:A4}. To see that note that if $f$ satisfies \ref{ass:A1}, \ref{ass:A2}, \ref{ass:A3'}  then we get  for all $t\in [0,(n+1)\tau]$ and $x,x_1,x_2,z\in\mathbb{R}^d$ that
{
\begin{equation}
    |f(t,x,z)|\leq (\bar  K+L)(1+|x|)(1+|z|),
\end{equation}
\begin{equation}
    |f(t,x_1,z)-f(t,x_2,z)|\leq L|x_1-x_2|,
\end{equation}

and
\begin{equation}\label{eqn:f_fulldiff}
    |f(t_1,x_1,z_1)-f(t_2,x_2,z_2)|\leq L\big( (1+|x|+|z|)|t_1-t_2|^\gamma+|x_1-x_2|+|z_1-z_2|^{\alpha}\big),
\end{equation}}
since $1+|x|+|z|\leq (1+|x|)(1+|z|)$ and $|z|^{\alpha}\leq 1+|z|$ for all $x,z\in\mathbb{R}^d$.  {Hence, the assumptions \ref{ass:A1}-\ref{ass:A4} are satisfied with $K=\bar K+L < \infty $, and under the assumptions  \ref{ass:A1}, \ref{ass:A2}, \ref{ass:A3'} the thesis of Theorem \ref{sol_dde_prop_1} holds. }
{\color{cyan}
}
\end{rem}
\begin{lem}\label{lem:gjholder}
    The  function $[0,\tau]\ni t\mapsto g_{j}(t,\phi_{j}(t))$ is bounded in $L^p([0,\tau])$ norm for each $j\in \mathbb{N}\setminus\{0\}$, and is $\min\{\gamma,\alpha\}$-H\"older continuous. 

    In particular, if the initial condition $\varphi(t)=x_0$ for $t\in [-\tau,0]$, then $g_0(\cdot, \phi_0(\cdot))$ is $\gamma$-H\"older continuous.
\end{lem}
\begin{proof}
    {Since the function $[0,\tau]\ni t\mapsto g_{j}(t,\phi_{j}(t))$ is Borel measurable and by Theorem \ref{sol_dde_prop_1}, \eqref{eqn:f_fulldiff}, \eqref{upper_est_Kj} and \eqref{phi_j_Lipschitz} we get}
\begin{equation}
\label{equ:supg}
    \|g_{j}(\cdot,\phi_{j}(\cdot))\|_{L^p([0,\tau])}\leq (1+K_{j-1})(1+K_{j})K<+\infty,
\end{equation}
and
\begin{align} \label{euq:diffgsandt}
&\left|g_j(s,\phi_j(s))-g_j(t,\phi_j(t))\right|\\
&\leq  L((1+|\phi_j(s)|+|\phi_{j-1}(s)|)|s-t|^\gamma+|\phi_{j}(s)-\phi_{j}(t)|+|\phi_{j-1}(s)-\phi_{j-1}(t)|^\alpha)\notag\\
&\leq  L((1+|K_j|+|K_{j-1}|)|s-t|^\gamma+|(1+K_{j-1})(1+K_{j}) K|t-s|\notag\\
&+|(1+K_{j-2})(1+K_{j-1}) K|t-s|^\alpha)\notag\\
&\leq {C_{g,j}}|s-t|^{\min\{\gamma,\alpha\}},
\end{align}
where {$C_{g,j}$} is a generic constant, for any $s,t\in [0,\tau]$ and $j \in \mathbb{N}$. 
\end{proof}
The main result of this section is as follows.

\begin{thm} 
\label{rate_of_conv_expl_Eul} 
Let $n\in\N\cup\{0\}$, $\tau\in (0,+\infty)$, $x_0\in\R^d$, and let $f$, ${\varphi}$ satisfy the assumptions  \ref{ass:A1}, \ref{ass:A2}, \ref{ass:A3'} for some $p\in [2,+\infty)$ and {$\alpha,\gamma\in (0,1]$}. There exist ${C_{p,0},C_{p,1},\ldots,C_{p,n}\in (0,+\infty)}$ such that for  all $N\geq \lceil \tau\rceil$ it holds

for $j=0,1,\ldots,n$,
	\begin{equation}
	\label{error_main_rk_thm}
		\Bigl\|\max\limits_{0\leq i\leq N}|x(t_i^j)-y_i^j|\Bigl\|_{p}\leq  {C_{p,j}} h^{\alpha^j \rho},
	\end{equation}
 where $\rho:=\frac{1}{2}+\min\{\gamma,\alpha\}$.

\end{thm}

\begin{proof} {In the proof we use the following auxiliary notations: $\delta_{k+1}^j=h(k+\gamma_{k+1}^j)$ and $h_{k+1}^j=h\cdot \gamma_{k+1}^j$. Then $\delta_{k+1}^j$ is uniformly distributed in $(t_k^0,t_{k+1}^0)$ and $\theta_{k+1}^j=\delta_{k+1}^j+j\tau$ is uniformly distributed in $(t_k^j,t_{k+1}^j)$. The latter one $h_{k+1}^j$ is a random stepsize uniformly distributed in $[0,h]$.}

We start with $j=0$ and consider the initial-vale problem \eqref{eqODE_0}. We define the auxiliary randomized Runge-Kutta (ARRK) scheme as 
\begin{align}\label{aux_rk_1}
    \begin{split}
	        &\bar y_0^0= y^0_0=y^{-1}_N=\varphi(0),\\
	        &\bar y_{k+1}^0=\bar y_k^0+h\cdot g_0(\theta_{k+1}^0,\bar y_k^0+h_{k+1}^0\cdot g_0(t_k^0,y_k^0)\big), \ k=0,1,\ldots, N-1.      
    \end{split}
\end{align}
as $g_0(t,x)=f(t,x,\varphi(t-\tau))$, for all $k=0,\ldots,N$ we have that $\bar y_k^0=y_k^0$. That is to say, the ARRK scheme coincides with RRK scheme at $j=0$.

From Lemma \ref{lem:gjholder} it has been shown that $g_0$ is $\min\{\gamma,\alpha\}$-H\"older continuous in time. Moreover, by \ref{ass:A1}, \ref{ass:A2} and \ref{ass:A3'} we have that $g_0$ is Borel measurable, and for all $t\in [0,\tau]$ and  $x,y\in\R^d$ 
\begin{align}
    \begin{split}
        &|g_0(t,x)-g_0(t,y)|\leq L|x-y|. 
    \end{split}
\end{align}
Hence, by Theorem \ref{sol_dde_prop_1} and by using analogous arguments as in the proof of Theorem 5.2 in \cite{RKYW2017}, which guarantees \eqref{error_main_rk_thm},
where {$C_{p,0}$} does not depend on $N$. 

Let us now assume that there exists $l\in\{0,1,\ldots,n-1\}$ for which there exists  { $C_{p,l} \in (0,+\infty)$} such that for all $N\geq \lceil \tau\rceil$
\begin{equation}
\label{ind_assumpt_1}
    \Bigl\|\max\limits_{0\leq i\leq N}|\phi_l(t_i^0)-y_i^l|\Bigl\|_{p}\leq  { C_{p,l}} h^{\alpha^l \rho}.
\end{equation}

We consider the following initial-value problem
\begin{equation}
\label{eqODE_l_1}
\begin{cases}
\phi_{l+1}'(t)=g_{l+1}(t,\phi_{l+1}(t)), & t\in [0,\tau], \\
\phi_{l+1}(0)=\phi_l({\tau}), 
\end{cases}
\end{equation}
with $g_{l+1}(t,x)=f(t+(l+1)\tau,x,\phi_l(t))$.
From Lemma \ref{lem:gjholder} we know that $g_{l+1}(\cdot,\phi_{l+1}(\cdot))$ is $\min\{\gamma,\alpha\}$-H\"older continuous.

{We define the ARRK scheme as follows
\begin{align} \label{eqn:aux_arrk_l+1}
    \begin{split}             
	       &\bar y_0^{l+1}= y^{l+1}_0=y^{l}_N,\\
	        &\bar y_{k+1}^{l+1}=\bar y_k^{l+1}+
         h\cdot g_{l+1}\big(\delta_{k+1}^{l+1},\bar y_k^{l+1}+
         h_{k+1}^{l+1} \cdot g_{l+1}(t^{0}_k,\bar y_k^{l+1})\big), \ k=0,1,\ldots, N-1.  
    \end{split}
\end{align}
}

From the definition it follows that $\bar y_k^j$, $j=0,1,\ldots,n$, $k=0,1,\ldots,N$, is measurable with respect to the {$\sigma$-field} generated by 
\eqref{sig_field1}, as well as is $y_k^j$. Moreover,  $\bar y^{l+1}_i$ approximates $\phi_{l+1}$ at $t_i^0$, despite $\{\bar y_i^{l+1}\}_{i=0,1,\ldots,N}$ is not implementable. We use $\bar y^{l+1}_i$ only in order to estimate the error \eqref{error_main_rk_thm} of $y^{l+1}_i$, since  it holds
\begin{equation}\label{error_decomp}
\begin{aligned}
\Bigl\|\max\limits_{0\leq i\leq N}|\phi_{l+1}(t_i^0)-y_i^{l+1}|\Bigl\|_{p} &\leq \Bigl\|\max\limits_{0\leq i\leq N}|\phi_{l+1}(t_i^0)-\bar y_i^{l+1}|\Bigl\|_{p} \\
&\quad+\Bigl\|\max\limits_{0\leq i\leq N}|\bar y_i^{l+1}-y_i^{l+1}|\Bigl\|_{p}.
\end{aligned}
\end{equation}

Firstly, we estimate $\displaystyle{\Bigl\|\max\limits_{0\leq i\leq N}|\bar y_i^{l+1}-y_i^{l+1}|\Bigl\|_{p}}$. For $k\in\{1,\ldots,N\}$ we get
\begin{align*}
&\bar y^{l+1}_k-y^{l+1}_k \\
&= \sum\limits_{j=1}^k(\bar y^{l+1}_j-\bar y^{l+1}_{j-1})-\sum\limits_{j=1}^k(y^{l+1}_j-y^{l+1}_{j-1})\\
&=h\sum\limits_{j=1}^k\Bigl(f\big(\theta_j^{l+1},{\bar y_{j-1}^{l+1}+h_{j}^{l+1}  f(t^{l+1}_{j-1},\bar y_{j-1}^{l+1},\phi_{l}(t^{0}_{j-1}))},\phi_{l}(\delta_{j-1}^{l+1})\big)\\
&\quad -f\big(\theta_j^{l+1},{y_{j-1}^{l+1}+h_{j}^{l+1}  f(t^{l+1}_{j-1}, y_{j-1}^{l+1},y_{j-1}^{l}),y_{j-1}^{l}+h_{j}^{l+1}  f(t^{l}_{j-1}, y_{j-1}^{l},y_{j-1}^{l-1})}\big)\Bigr),
\end{align*}
which, by using \eqref{eqn:f_assum_xz} twice, gives
{
\begin{align*}
    &|\bar y^{l+1}_k-y^{l+1}_k |\\
&\leq h L \sum\limits_{j=1}^k \big( (1+hL)|\bar y^{l+1}_{j-1}-y^{l+1}_{j-1} |+hL|\phi_{l}(t^0_{j-1})-y_{j-1}^{l}|^\alpha\big)\\
&\quad +h L \sum\limits_{j=1}^k   |\phi_{l}(\delta_{j-1}^{l+1})-y^{l}_{j-1}-h_{j}^{l+1}  f(t^{l}_{j-1}, y_{j-1}^{l},y_{j-1}^{l-1}) |^{\alpha}\\
&\leq  h L(1+L) \sum\limits_{j=1}^k |\bar y^{l+1}_{j-1}-y^{l+1}_{j-1} |+h^2L^2\sum\limits_{j=1}^k|\phi_{l}(t_{j-1}^0)-y_{j-1}^{l}|^\alpha\\
&\quad +h L \sum\limits_{j=1}^k   \Big|\phi_{l}(t^0_{j-1})-y^{l}_{j-1}\\
&\qquad\qquad+\int_{t_{j-1}^l}^{t_{j-1}^l+h_{j}^{l+1}}\big(f(s,\phi_{l}(s-l\tau),\phi_{l-1}(s-l\tau) )-f(t^{l}_{j-1}, y_{j-1}^{l},y_{j-1}^{l-1})\big)\mathrm{d}s \Big|^{\alpha},
\end{align*}

where for the last term we use the expression of the initial-value problem \eqref{eqODE_l_1}. Note that $|a+b|^\alpha\leq |a|^\alpha+|b|^\alpha$ for $\alpha\in (0,1]$. Now using \eqref{eqn:f_assum_xz}, [regularity of f wrt time], and Theorem \ref{sol_dde_prop_1} to estimate the last term, one can obtain that
\begin{equation}\label{eqn:estimate_forj+1}
    \begin{split}
       &\Big|\phi_{l}(t^0_{j-1})-y^{l}_{j-1}+\int_{t_{j-1}^l}^{t_{j-1}^l+h_{j}^{l+1}}\big(f(s,\phi_{l}(s-l\tau),\phi_{l-1}(s-l\tau) )-f(t^{l}_{j-1}, y_{j-1}^{l},y_{j-1}^{l-1})\big)\mathrm{d}s \Big|^{\alpha}\\
   &\leq |\phi_{l}(t^0_{j-1})-y^{l}_{j-1}|^\alpha+h^\alpha \big|f(t^{l}_{j-1}, \phi_{l}(t^0_{j-1}),\phi_{l-1}(t^0_{j-1}))-f(t^{l}_{j-1}, y_{j-1}^{l},y_{j-1}^{l-1})\big|^\alpha \\
   &\quad+\Big|\int_{t_{j-1}^l}^{t_{j-1}^l+h_{j}^{l+1}}\big(f(s,\phi_{l}(s-l\tau),\phi_{l-1}(s-l\tau) )-f(t^{l}_{j-1}, \phi_{l}(t^0_{j-1}),\phi_{l-1}(t^0_{j-1})\big)\mathrm{d}s \Big|^{\alpha}\\
   &\leq |\phi_{l}(t^0_{j-1})-y^{l}_{j-1}|^\alpha+L^\alpha h^\alpha \big( | \phi_{l}(t^0_{j-1})-y_{j-1}^{l}|^\alpha+|\phi_{l-1}(t^0_{j-1})-y_{j-1}^{l-1}|^{\alpha^2}\big)\\
   &\quad +L^\alpha h^\alpha \big((1+K_l+K_{l-1})^\alpha h^{\gamma \alpha}+K^\alpha (1+K_{l})^\alpha(1+K_{l-1})^\alpha h^{\alpha}\\
   &\qquad \qquad \qquad+K^{\alpha^2} (1+K_{l-1})^{\alpha^2}(1+K_{l-2})^{\alpha^2} h^{\alpha^2}\big).
    \end{split}
\end{equation}

Thus, overall,
\begin{align*}
        &|\bar y^{l+1}_k-y^{l+1}_k |\\
        &\leq h L(1+L) \sum\limits_{j=1}^k |\bar y^{l+1}_{j-1}-y^{l+1}_{j-1} |+Lh(1+L^\alpha h^\alpha+Lh)\sum\limits_{j=1}^k|\phi_{l}(t^0_{j-1})-y_{j-1}^{l}|^\alpha\\
        &\quad +L^{1+\alpha} h^{1+\alpha} \sum\limits_{j=1}^k  |\phi_{l-1}(t^0_{j-1})-y_{j-1}^{l-1}|^{\alpha^2}\\
        &\quad+L^{1+\alpha }h^{1+\alpha}(1+K)^\alpha \Big(\prod_{m=l-2}^l (1+K_{m})^\alpha\Big)\sum\limits_{j=1}^k \big(h^{\gamma \alpha}+h^{\alpha}+h^{\alpha^2})\\
        &\leq h L(1+L) \sum\limits_{j=1}^k \max\limits_{0\leq i\leq j-1} |\bar y^{l+1}_i-y^{l+1}_i| + L\tau (1+2L)\max\limits_{0\leq i\leq N}|\phi_{l}(t_i^0)-y_i^l|^{\alpha}\\
        &\quad + L^{1+\alpha} \tau h^\alpha \max\limits_{0\leq i\leq N}|\phi_{l-1}(t_i^0)-y_i^{l-1}|^{\alpha^2}\\
        &\quad+3L^{1+\alpha }\tau (1+K)^\alpha \Big(\prod_{m=l-2}^l (1+K_{m})^\alpha\Big){h^{\alpha (\min\{\gamma, \alpha\} +1)}}. 
\end{align*}

Taking the $L_p$-norm on both sides gives that
\begin{align}\label{eqn:error2lp}
\begin{split}
        &\mathbb{E}\Bigl(\max\limits_{0\leq i\leq k}|\bar y^{l+1}_i-y^{l+1}_i|^p\Bigr)\leq c_p h^p L^p\mathbb{E}\Bigl(\sum\limits_{j=1}^k \max\limits_{0\leq i\leq j-1} |\bar y^{l+1}_i-y^{l+1}_i|\Bigr)^p \\
    &\qquad +c_p  L^p\tau^p (1+2L)^p\mathbb{E}\Bigl[\max\limits_{0\leq i\leq N}|\phi_{l}(t_i^0)-y_i^l|^{\alpha p}\Bigr] \\
       &\qquad +c_p   L^{(1+\alpha)p} \tau^p h^{\alpha p}\mathbb{E}\Bigl[\max\limits_{0\leq i\leq N}|\phi_{l-1}(t_i^0)-y_i^{l-1}|^{\alpha^2 p}\Bigr] \\ 
       &\qquad +c_p { 3^p} L^{(1+\alpha)p }\tau^p (1+K)^{\alpha p} \Big(\prod_{m=l-2}^l (1+K_{m})^\alpha\Big)^p {h^{\alpha (\min\{\gamma, \alpha\} +1)p}}.
\end{split}
\end{align}

By Jensen inequality and \eqref{ind_assumpt_1} we have 
\begin{equation*}
    \mathbb{E}\Bigl[\max\limits_{0\leq i\leq N}|\phi_{l}(t_i^0)-y_i^l|^{\alpha p}\Bigr]\leq\Biggl( \mathbb{E}\Bigl[\max\limits_{0\leq i\leq N}|\phi_{l}(t_i^0)-y_i^l|^{p}\Bigr]\Biggr)^{\alpha}\leq  {C_{p,l}^{\alpha p}} h^{\alpha^{l+1}\rho p},
\end{equation*}
and 
\begin{equation*}
    \mathbb{E}\Bigl[\max\limits_{0\leq i\leq N}|\phi_{l-1}(t_i^0)-y_i^{l-1}|^{\alpha^2 p}\Bigr]\leq {C_{p, l-1}^{\alpha^2 p}} h^{\alpha^{l+1}\rho p}.
\end{equation*}

Note that via the H\"older inequality we get, for $k=1,2,\ldots,N$ and all positive numbers $a_1,a_2,\ldots,a_k$, that
\begin{equation}
\label{h_ineq_app}
    h^p\Bigl (\sum\limits_{j=1}^k a_k\Bigr)^p\leq\tau^{p-1} h \sum\limits_{j=1}^k a_j^p,
\end{equation}

and hence
\begin{equation}\label{est_2}
    h^p\mathbb{E}\Bigl[\sum\limits_{j=1}^k \max\limits_{0\leq i\leq j-1} |\bar y^{l+1}_i-y^{l+1}_i|\Bigr]^p 
    \leq\tau^{p-1} h  \sum\limits_{j=1}^k\mathbb{E}\Bigl[ \max\limits_{0\leq i \leq j-1}|\bar y_i^{l+1}-y_i^{l+1}|^p \Bigr].
\end{equation}
Finally, substituting all the estimates above into \eqref{eqn:error2lp} gives
\begin{align}\label{eqn:error2_final2}
    \begin{split}
                &\mathbb{E}\Bigl(\max\limits_{0\leq i\leq k}|\bar y^{l+1}_i-y^{l+1}_i|^p\Bigr)\\
                &\leq c_p \tau^{p-1} L^p h\sum\limits_{j=1}^k\mathbb{E}\Bigl[ \max\limits_{0\leq i \leq j-1}|\bar y_i^{l+1}-y_i^{l+1}|^p \Bigr] +{\bar C_{1,p,l+1}} h^{\alpha^{l+1}\rho p} .
    \end{split}
\end{align}

{where $\bar C_{1,p,l+1}$ is a generic constant, for any $l\in\{0,1,\ldots,n-1\}, p\in [2,\infty)$}. \\
Now applying Gronwall inequality to \eqref{eqn:error2_final2} gives that
\begin{equation}
    \mathbb{E}\Bigl(\max\limits_{0\leq i\leq k}|\bar y^{l+1}_i-y^{l+1}_i|^p\Bigr)\leq   {C_{1,p,l+1}} h^{\alpha^{l+1}\rho p}.
    \label{est_diff_ytoybar_l+1}
\end{equation}
}

We now establish an upper bound on  $\displaystyle{\Bigl\|\max\limits_{0\leq i\leq N}|\phi_{l+1}(t_i^0)-\bar y_i^{l+1}|\Bigl\|_{p}}$. For $k\in\{1,2,\ldots,N\}$ we have
\begin{align}
\begin{split}
          &\phi_{l+1}(t_k^0)-\bar y^{l+1}_k\\
        =&\phi_{l+1}(0)-\bar y_0^{l+1}+(\phi_{l+1}(t_k^0)-\phi_{l+1}(t_0^0))-(\bar y^{l+1}_k-\bar y^{l+1}_0)\\
        =&(\phi_l(t_N^0)-y_N^l)+\sum\limits_{j=1}^k(\phi_{l+1}(t_j^0)-\phi_{l+1}(t_{j-1}^0))-\sum\limits_{j=1}^k(\bar y^{l+1}_j-\bar y^{l+1}_{j-1})\\
        =&(\phi_l(t_N^0)-y_N^l)+\sum\limits_{j=1}^k\Bigl(\int\limits_{t_{j-1}^0}^{t_{j}^0}g_{l+1}(s,\phi_{l+1}(s))\de s \\
        &~~~~~~~~~~~~~~~~~~~~~~~
        -h\cdot g_{l+1}\big(\delta_{j}^{l+1},\bar y_{j-1}^{l+1}+
         h_{j}^{l+1} \cdot g_{l+1}(t^{0}_{j-1},\bar y_{j-1}^{l+1})\big)\Bigr)\\
        =&(\phi_l(t_N^0)-y_N^l)+S^k_{1,l+1}+S^k_{2,l+1}+S^k_{3,l+1},
\end{split}
\end{align}
where
{
\begin{align*}
S^k_{1,l+1} 
&\udef \sum\limits_{j=1}^k \left(\int\limits_{t_{j-1}^0}^{t_{j}^0}g_{l+1}(s,\phi_{l+1}(s))\de s-h\cdot  g_{l+1}(\delta_j^{l+1},\phi_{l+1}(\delta_j^{l+1}))\right), \\
S^k_{2,l+1} 
&\udef h\sum\limits_{j=1}^k\Bigl(g_{l+1}(\delta_j^{l+1},\phi_{l+1}(\delta_j^{l+1}))\\
&~~~~~~~~~~
-g_{l+1}(\delta_j^{l+1}, \phi_{l+1}(t^0_{j-1})+h_{j}^{l+1} g_{l+1}(t^0_{j-1},\phi_{l+1}(t^0_{j-1})))\Bigr), \\
S^k_{3,l+1}
&\udef h\sum\limits_{j=1}^k\Bigl(
g_{l+1}(\delta_j^{l+1}, \phi_{l+1}(t^0_{j-1})+h_{j}^{l+1} g_{l+1}(t^0_{j-1},\phi_{l+1}(t^0_{j-1}))) \\
&~~~~~~~~~~
-g_{l+1}(\delta_j^{l+1},\bar y_{j-1}^{l+1}+h_{j}^{l+1} g_{l+1}(t_{j-1}^{0},\bar y_{j-1}^{l+1}))\Bigr).
\end{align*}}
Since $g_{l+1}(\cdot,\phi_{l+1}(\cdot))$ is $\min\{\gamma,\alpha\}$-H\"older continuous from Lemma \ref{lem:gjholder} and by \eqref{equ:supg} and   \eqref{euq:diffgsandt} we get that
{
\begin{equation}
\label{est_S1}
        \Bigl\|\max\limits_{1\leq k\leq N}|S^k_{1,l+1}|\Bigl\|_{p} \leq c_{p}\sqrt{\tau}\|g_{l+1}(\cdot,\phi_1(\cdot))\|_{C^{\min\{\gamma,\alpha\}}([0,\tau])} {h^\rho}=:
        {C_{S_1,l+1}}
        {h^\rho}.
\end{equation}}
Then, by above inequity and Theorem \ref{sol_dde_prop_1} we get
{
\begin{align}
       |S^k_{2,l+1}|&\leq Lh\sum\limits_{j=1}^k \Bigl|\phi_{l+1}(\delta_j^{l+1})-\phi_{l+1}(t_{j-1}^0) - h_{j}^{l+1} g_{l+1}(t^0_{j-1},\phi_{l+1}(t^0_{j-1}))\Bigr| \notag\\
&\leq Lh\sum\limits_{j=1}^k \int^{t^0_{j-1}+h_{j}^{l+1}}_{t^0_{j-1}}\Bigl|g_{l+1}(s,\phi_{l+1}(s)) - g_{l+1}(t^0_{j-1},\phi_{l+1}(t^0_{j-1}))\Bigr| \de s \notag\\
& \leq Lh\sum\limits_{j=1}^k \int^{t^0_{j-1}+h_{j}^{l+1}}_{t^0_{j-1}}\Bigl|C_{g(\cdot),{l+1}}|s-t^0_{j-1}|^{\min\{\gamma,\alpha\}}\Bigr| \de s \notag\\
& \leq C_{g(\cdot),{l+1}} Lh\sum\limits_{j=1}^k {\min\{\gamma,\alpha\}}^{-1}\cdot (h_{j}^{l+1})^{\min\{\gamma,\alpha\} + 1}\notag \\
& \leq C_{g(\cdot),{l+1}} {\min\{\gamma,\alpha\}}^{-1} LhN h^{\min\{\gamma,\alpha\} + 1}\notag \\
& {=: C_{S_{2}, l+1}} h^{\min\{\gamma,\alpha\} + 1}. 
        \label{est_S2}
\end{align}}

Moreover, for the last term $ S^k_{3,l+1}$, by using \eqref{eqn:f_fulldiff} , we get 
\begin{equation}
\label{est_S3}
\begin{aligned}
S^k_{3,l+1} 
&\leq hL\sum\limits_{j=1}^k\Bigl(
\phi_{l+1}(t^0_{j-1})+h_{j}^{l+1} g_{l+1}(t^0_{j-1},\phi_{l+1}(t^0_{j-1}))-
\bar y_{j-1}^{l+1}-h_{j}^{l+1} g_{l+1}(t_{j-1}^{0},\bar y_{j-1}^{l+1})\Bigr)\\
&\leq h L \sum\limits_{j=1}^k\Bigl(
\phi_{l+1}(t^0_{j-1})-
\bar y_{j-1}^{l+1} + h_{j}^{l+1} L (\phi_{l+1}(t^0_{j-1}) - \bar y_{j-1}^{l+1})\Bigr)\\
&\leq h L \sum\limits_{j=1}^k\Bigl(
 (h_{j}^{l+1} L +1)(\phi_{l+1}(t^0_{j-1}) - \bar y_{j-1}^{l+1})\Bigr)\\
 &\leq h L  (L +1)\sum\limits_{j=1}^k\Bigl(
(\phi_{l+1}(t^0_{j-1}) - \bar y_{j-1}^{l+1})\Bigr) 
\end{aligned}
\end{equation}

{Therefore, for all $k\in\{1,2,\ldots,N\}$ the following inequality holds
\begin{align}
\mathbb{E}\Bigl[\max\limits_{0\leq i\leq k}|\phi_{l+1}(t_i^0)-\bar y^{l+1}_i|^p\Bigr] & \leq \mathbb{E}\Bigl[
\max\limits_{0\leq i\leq k}|
(\phi_l(t_N^0)-y_N^l)+S^k_{1,l+1}+S^k_{2,l+1}+S^k_{3,l+1}
|^p
]\notag\\
&\leq {c_p}\Bigl(\mathbb{E}|\phi_l(t_N^0)-y_N^l|^p + \mathbb{E}\Bigl[\max\limits_{1\leq k\leq N}|S^k_{1,l+1}|^p\Bigr]+\mathbb{E}\Bigl[\max\limits_{1\leq k\leq N}|S^k_{2,l+1}|^p\Bigr]\Bigr) \notag \notag\\
&+c_p \mathbb{E}\Bigl[\max\limits_{1\leq k\leq N}|S^k_{3,l+1}|^p\Bigr] \notag
\\&\leq {c_p}\Bigl({C_{p,l}^p} h^{\alpha^l \rho p} + {C_{S_1,l+1}^{p} h^{p\rho}}+ { C_{S_2,l+1}^{p} h^{p(\min\{\gamma,\alpha\} + 1)}} \Bigr)\notag\\
&+{c_p} \mathbb{E}\Bigl[ (h L (L+1) \sum\limits_{j=1}^k |\phi_{l+1}(t_{j-1}^0)-\bar y^{l+1}_{j-1}|)^p \Bigr] \label{equ:S3step2label1}
\\&\leq {\bar C_{2,p,l+1}} h^{\alpha^l \rho p} +
c_p^{l+1} L^p (L+1)^p \tau^{p-1} h  \sum\limits_{j=1}^k \mathbb{E}\Bigl[|\phi_{l+1}(t_{j-1}^0)-\bar y^{l+1}_{j-1}|^p \Bigr]. \label{equ:S3step2label2},
\end{align}}

{where $\bar C_{2,p,l+1}$ is a generic constant, for any $l\in\{0,1,\ldots,n-1\}, p\in [2,\infty)$}.\\
{In inequality \eqref{equ:S3step2label1}, we utilize \eqref{ind_assumpt_1}, \eqref{est_S1}, \eqref{est_S2}, and \eqref{est_S3}. Then, in inequality \eqref{equ:S3step2label2}, we obtain the result by invoking \eqref{h_ineq_app}.}
{
By using  Gronwall's lemma,  we get for all $k\in\{1,2,\ldots,N\}$
\begin{align*}
\mathbb{E}\Bigl[\max\limits_{0\leq i\leq k}|\phi_{l+1}(t_i^0)-\bar y^{l+1}_i|^p\Bigr] 
& \leq {\bar C_{2,p,l+1}} h^{\alpha^l \rho p}\exp(c_p L^p (L+1)^p \tau^{p-1} h N)\\
& \leq {\bar C_{2,p,l+1}} h^{\alpha^l \rho p}\exp(c_p L^p (L+1)^p \tau^{p}),
\end{align*}

which gives
\begin{equation}
\label{error_phi_byk2}
    \Bigl\|\max\limits_{0\leq i\leq N}|\phi_{l+1}(t_i^0)-\bar y_i^{l+1}|\Bigl\|_{p}\leq { C_{2,p,l+1}} h^{\alpha^l \rho}.
\end{equation}}
Combining \eqref{error_decomp}, \eqref{est_diff_ytoybar_l+1}, and \eqref{error_phi_byk2} we finally obtain
{
\begin{equation}
     \Bigl\|\max\limits_{0\leq i\leq N}|\phi_{l+1}(t_i^0)- y_i^{l+1}|\Bigl\|_{p}\leq { C_{p,l+1}} h^{\alpha^{l+1} \rho},
\end{equation}}

which ends the inductive part of the proof. Finally,  $\phi_{l+1}(t_i^0)=\phi_{l+1}(ih)=x(ih+(l+1)\tau)=x(t_i^{l+1})$ and the proof of \eqref{error_main_rk_thm} is finished. 
\end{proof}

\section{Numerical experiments}
\subsection{The implementation}
The implementation of the randomized Runge-Kutta method is not straightforward. To evaluate the solution at time point $t^j_i=ih+j\tau$ within the interval $[j\tau,(j+1)\tau]$ for $j\geq 1$, we need four steps:

{\bf Step (j1):} First simulate $\gamma \sim \mathcal{U}(0,1)$ and set a random stepsize $\gamma h\in [0,h]$ and a random time point $
    t^j_{i-1}+ \gamma h\in [t^j_{i-1},t^j_i]$.
    
{\bf Step (j2):} Following the second line of \eqref{eqn:rrk_j}, compute the intermediate delay term $\tilde y_{i}^{j-1,j}$ via the random stepsize, the time grid point and evaluations from the preceding $\tau$-intervals, namely, $t^{j-1}_{i-1}$ and $y_{i-1}^{j-1}$ over $[(j-1)\tau,j\tau]$, and $y_{i-1}^{j-2}$ within $[(j-2)\tau,(j-1)\tau]$.

{\bf Step (j3):} Following the third line of \eqref{eqn:rrk_j}, compute the intermediate term $\tilde y_{i}^{j}$ via the random stepsize, the time grid point and previous evaluations from the preceding or the current $\tau$-interval, namely, $t^{j}_{i-1}$ and $y_{i-1}^{j}$ within $[j\tau,(j+1)\tau]$, and $y_{i-1}^{j-1}$ over $[(j-1)\tau,j\tau]$.

{\bf Step (j4):} Following the last line of \eqref{eqn:rrk_j}, compute evaluation $y_{i}^{j}$ via the random time point, the preceding evaluation $y_{i-1}^{j}$, the intermediate term $\tilde y_{i}^{j}$ and the intermediate delay term $\tilde y_{i}^{j-1,j}$ from {\bf Step (j1)-Step (j3)}.

\begin{lstlisting}[caption={A sample implementation of \eqref{eqn:rrk_0} and \eqref{eqn:rrk_j} in
\textsc{Python}}, label=py:RandEM]
import numpy as np

def f(t,x,z):
    return [...]

def randEM_full(tau, initial_func, h, f,n_taus):
    # input:    delay lag tau, the initial function over [-tau, 0],  
    #            stepsize h, drift function f,
    #            number of intervals of length tau n_taus
    # output:   
    #  one trajectory of the randomized Runge-Kutta method

    ###to get numerical evaluation:

    #number of steps within one interval with length tau
    N=int(tau/h) 
                    
    #Initiate the solution 'matrix', where the first column 
    #corresponds to the delay conditions over [-tau, 0] 
    sol = np.zeros((N+ 1,n_taus+1))
    
    #setting the initial conditions on grid point over [-tau, 0] 
    grid_delay = - tau + h *np.arange(0, N) 
    sol[:,0] = np.array([initial_func(grid_delay[i])\
                         for i in range(len(grid_delay))])

    #for each interval [(j-1)*tau, j*tau]     
    for j in range(1, n_taus+1):   
        
        # collect the grid points over 	[(j-1)*tau, j*tau]  
        grid = (j-1)* tau + h *np.arange(0, N)  
        
        sol[0,j] = sol[-1,j-1]
        
        if j-1==0:
            for i in range(1,N+1):

                # step (j1):
                gamma=np.random.rand() 
                rand_time=grid[i-1]+gamma*h
                rand_h=gamma*h
            
                # step (j2):
                sol_delay_inter=initial_func(rand_time-tau)

                # step (j3):
                drift_inter=rand_h*f(grid[i-1], sol[i-1,j],\
                                     initial_func(grid[i-1]-tau))
                sol_inter = sol[i-1,j]+drift_inter 
            
                # step (j4):
                drift=h*f(rand_time, sol_inter, sol_delay_inter)
                sol[i,j] = sol[i-1,j]+drift               
        else:         
            for i in range(1,N+1):

                # step (j1):
                gamma=np.random.rand() 
                rand_time=grid[i-1]+gamma*h
                rand_h=gamma*h

                # step (j2):
                drift_delay=rand_h*f(grid[i-1]-tau, sol[i-1,j-1],\
                                     sol[i-1,j-2])
                sol_delay_inter = sol[i-1,j-1]+drift_delay

                # step (j3):
                drift_inter=rand_h*f(grid[i-1], sol[i-1,j],\
                                     sol[i-1,j-1])
                sol_inter = sol[i-1,j]+drift_inter 

                # step (j4):
                drift=h*f(rand_time, sol_inter, sol_delay_inter)
                sol[i,j] = sol[i-1,j]+drift 
    return sol
\end{lstlisting}

\subsection{Example 1} We modify \cite[Example 6.2]{kruse2017randomized} as follows:
 \begin{align}
  \label{eqn:DDEexample1}
  \begin{split}
    \begin{cases}
      \dot{u}(t) &= g(t)\cdot\big(u(t)+(1+|u(t-\tau)|)^\alpha\big), \quad t \in [0,3\tau],\\
      u(t) &= 1, \quad t \in [-\tau, 0],
    \end{cases}
  \end{split}
\end{align}

where  $g(t):=\left[- \frac{1}{10} \mbox{sgn}(\frac{1}{4}
T-t)-\frac{1}{5} \mbox{sgn}(\frac{1}{2} T-t)- \frac{7}{10}
\mbox{sgn}(\frac{3}{4}T-t)\right]$\ and 
 \begin{align}
  \mbox{sgn}(t):=
  \begin{split}
    \begin{cases}
      -1,& \ \mbox{if\ }t<0,\\
       0,& \ \mbox{if\ }t=0,\\      
       1,& \ \mbox{if\ }t>0.
    \end{cases}
  \end{split}
\end{align}

Here we have three jump points at $t=\frac{1}{4} T$, $t=\frac{1}{2} T$ and
$t=\frac{3}{4}T$. Note that when $\alpha=0$, our proposed method can recover the performance shown in \cite[Example 6.2]{kruse2017randomized}, with an order of convergence roughly $1.5$. 

In this example, we choose $\alpha=0.5$ and investigate the performance of our proposed method randomized Runge-Kutta method for solving such time-irregular DDE. The performance is further compared with the one of the randomized Euler method proposed in \cite{difonzo2024existence}, which is known to efficiently handle Carath\'eodory DDEs, in terms of convergence rate, accuracy, and computational efficiency. The reference solution is computed through the randomized Runge-Kutta method at stepsize $h_{\text{ref}}=2^{-15}$. 
We test the performance via $1000$ experiments for each $h=2^l$, $l=2,\ldots,7$, and record the mean-square error (MSE) and time cost against MSE for both methods respectively in Figure \ref{fig:f62}. 

It can be observed from Figure \ref{fig:alpha0.5,f62,rrk} and Figure \ref{fig:alpha0.5,f62,re} that the randomized Runge-Kutta method outperforms consistently over all intervals in terms of accuracy and the order of convergence. Clearly, due to the additional computation of $\tilde y_{k+1}^j$ and $\tilde y_{k+1}^{j,j-1}$ in \eqref{eqn:rrk_j}, the randomized Runge-Kutta method
is approximately three times as expensive as the randomized Runge-Kutta method with the same step size. But due to its better accuracy, the randomized Runge-Kutta method is superior for all the step sizes smaller than $2^{-3}$.
\begin{figure}
\centering
\begin{subfigure}{0.45\textwidth}
    \centering
    \includegraphics[width=\textwidth]{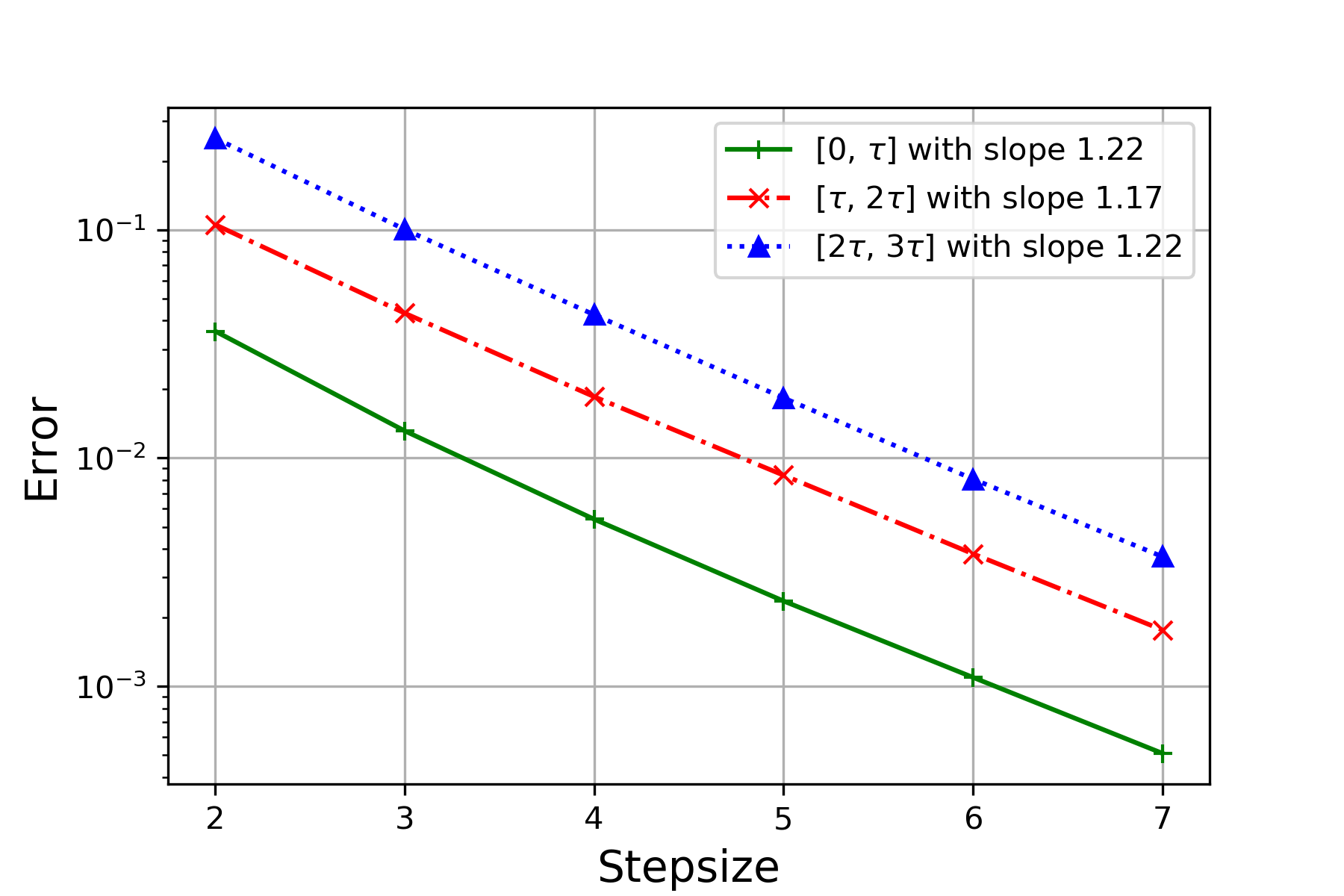}
    \caption{MSE of randomized Runge-Kutta.}
    \label{fig:alpha0.5,f62,rrk}
\end{subfigure}
\begin{subfigure}{0.45\textwidth}
    \centering
    \includegraphics[width=\textwidth]{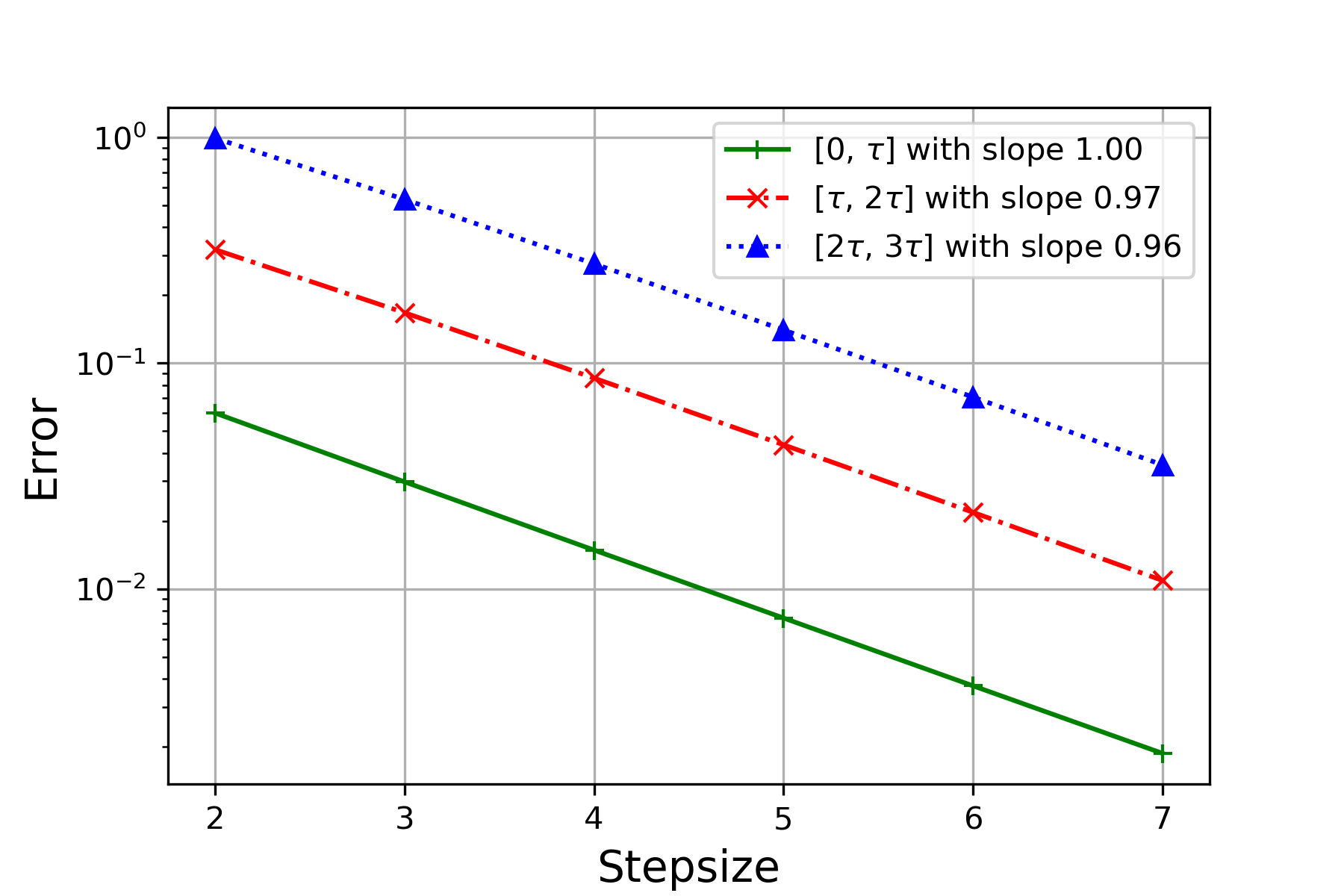}
    \caption{MSE of randomized Euler.}
    \label{fig:alpha0.5,f62,re}
\end{subfigure}\\
\begin{subfigure}{0.45\textwidth}
    \centering
    \includegraphics[width=\textwidth]{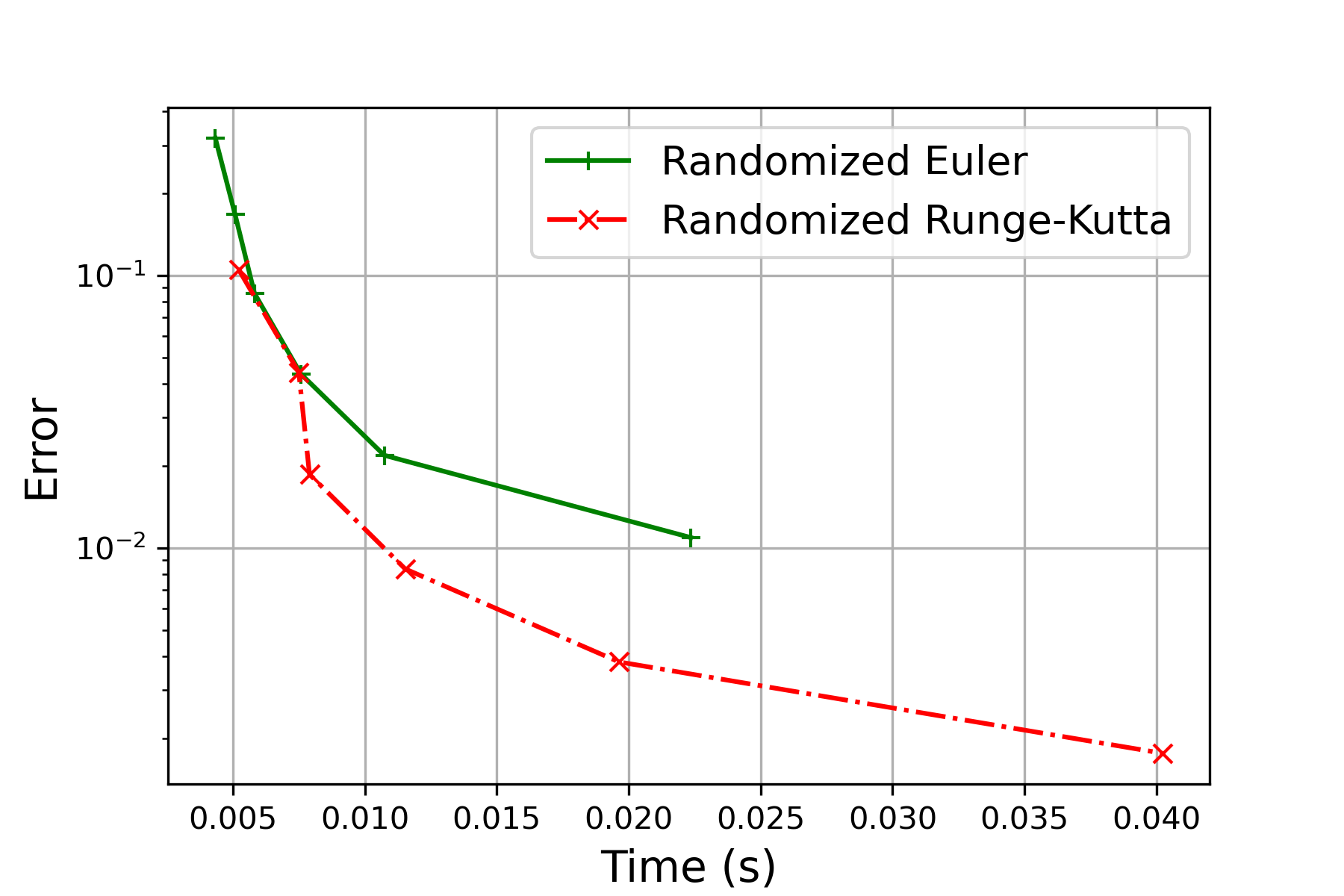}
    \caption{Time cost against MSE over $[\tau, 2\tau]$.}
    \label{fig:alpha0.5,f62,time,interval1}
\end{subfigure} \begin{subfigure}{0.45\textwidth}
    \centering
    \includegraphics[width=\textwidth]{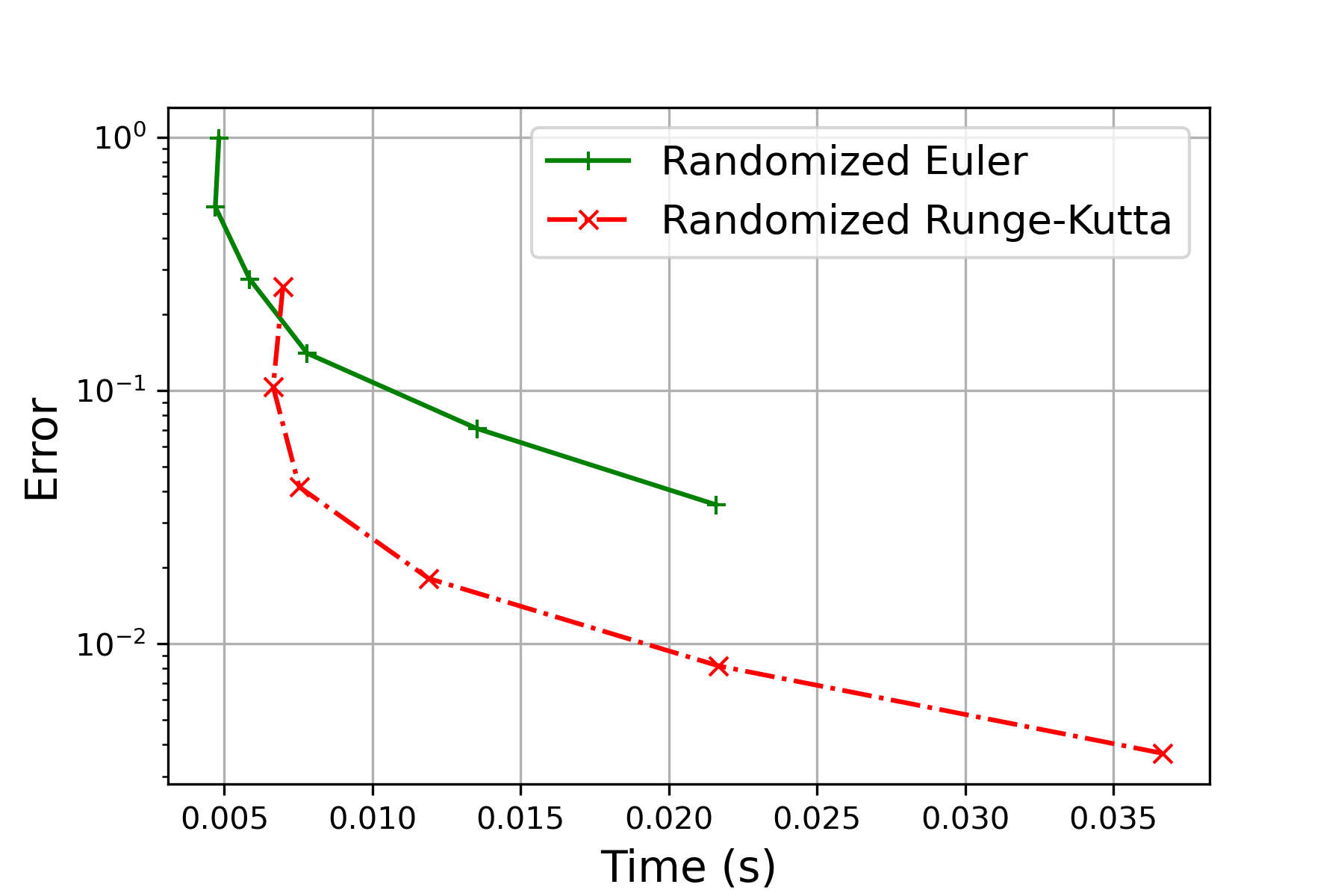}
    \caption{Time cost against MSE over $[2\tau, 3\tau]$.}
    \label{fig:alpha0.5,f62,time,interval2}
\end{subfigure}
\caption{The performance of solving DDE \eqref{eqn:DDEexample1} via randomized Runge-Kutta and randomised Euler.}
\label{fig:f62}
\end{figure}
\subsection{Example 2} To verify  \eqref{error_main_rk_thm} in Theorem \ref{rate_of_conv_expl_Eul}, we test the sensitivity of the performance of the randomized Runge-Kutta method with respect to parameters $\alpha, \gamma\in (0,1]$ on the following DDE:
 \begin{align}
  \label{eqn:DDEexample2}
  \begin{split}
    \begin{cases}
      \dot{u}(t) &= u(t)-|u(t-\tau)|^\alpha+|t|^\gamma, \quad t \in [0,3\tau],\\
      u(t) &= t+\tau, \quad t \in [-\tau, 0].
    \end{cases}
  \end{split}
\end{align}

Clearly DDE \eqref{eqn:DDEexample2}
satisfies all the assumptions of Theorem \ref{rate_of_conv_expl_Eul}. We choose the combinations 
$$(\alpha,\gamma)\in \{(0.1,0.1), (0.1,0.5), (0.5,0.1), (0.5,0.5), (0.5,1), (1,0.5)\}$$

For each pair, the reference solution is computed through the randomized Runge-Kutta method at stepsize $h_{\text{ref}}=2^{-16}$. 
We test the performance via $1000$ experiments for each $h=2^l$, $l=5,\ldots,10$, and record the mean-square error (MSE) and the negative MSE slope in Figure \ref{fig:newexample} and Table \ref{tab:DDEexample2} respectively. One can observe from Table \ref{tab:DDEexample2} that for each pair of $(\alpha, \gamma)$, the order of convergence over $[j\tau, (j+1)\tau]$ is consistently higher than the theoretical one for $j=0,1,2$, where the latter one is $\Big(\frac{1}{2}+(\gamma \wedge \alpha)\Big)\alpha^j$ in \eqref{error_main_rk_thm}. Particularly, it can been seen from Figure \ref{fig:alpha0.1,c0.1,new}, Figure \ref{fig:alpha0.5,c0.1,new} and  Figure \ref{fig:alpha0.1,c0.5,new} that, though orders of convergence from the three cases are similar over $[0,\tau]$, the ones over $[\tau,2\tau]$ and $[2\tau,3\tau]$ of Figure \ref{fig:alpha0.5,c0.1,new} are slightly higher than the other two, due to the larger value of $\alpha$. The result coincides with Theorem \ref{rate_of_conv_expl_Eul}.

\begin{figure}
\centering
\begin{subfigure}{0.32\textwidth}
    \centering
    \includegraphics[width=\textwidth]{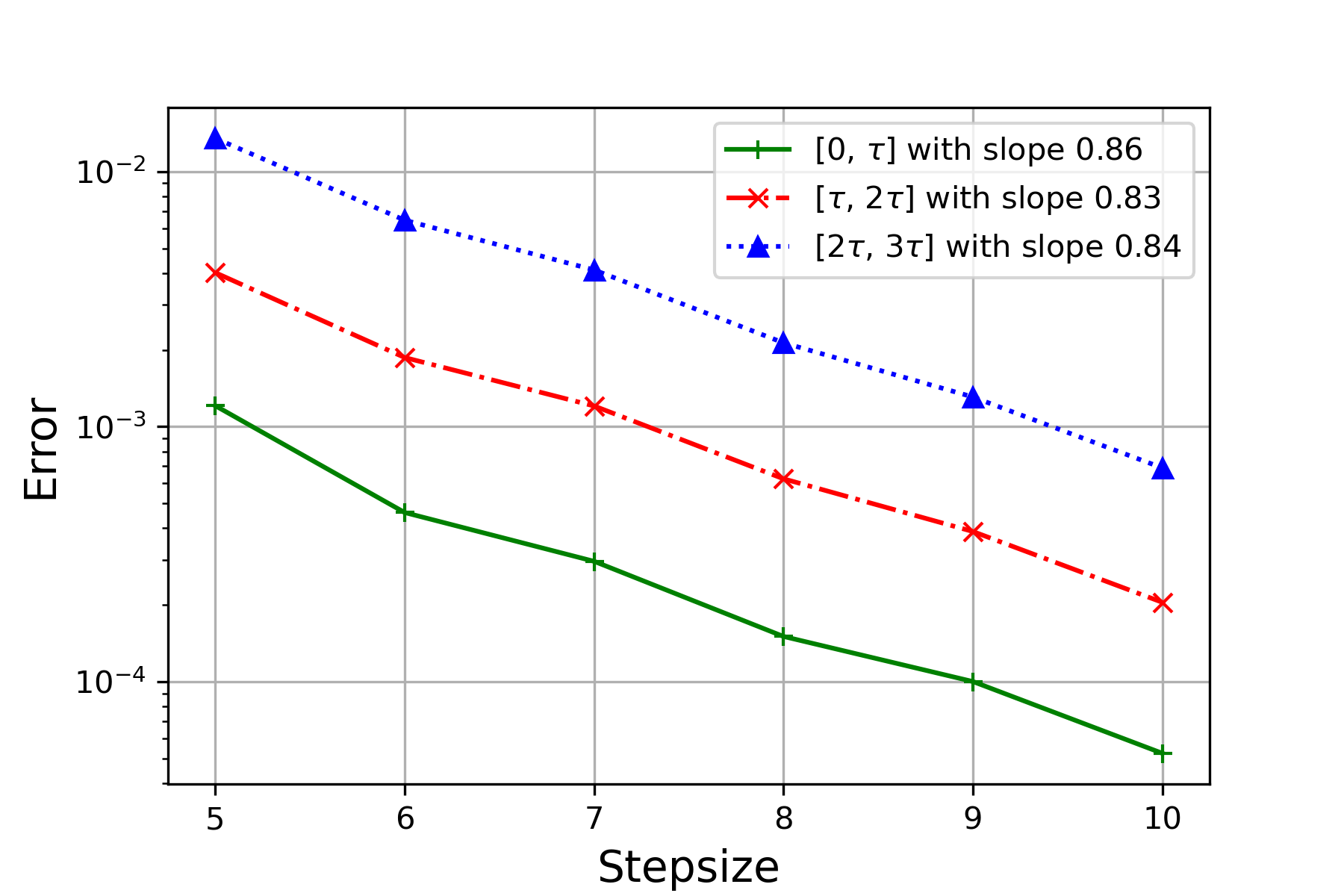}
    \caption{$(\alpha,\gamma)=(0.1,0.1)$}
    \label{fig:alpha0.1,c0.1,new}
\end{subfigure}
\begin{subfigure}{0.31\textwidth}
    \centering
    \includegraphics[width=\textwidth]{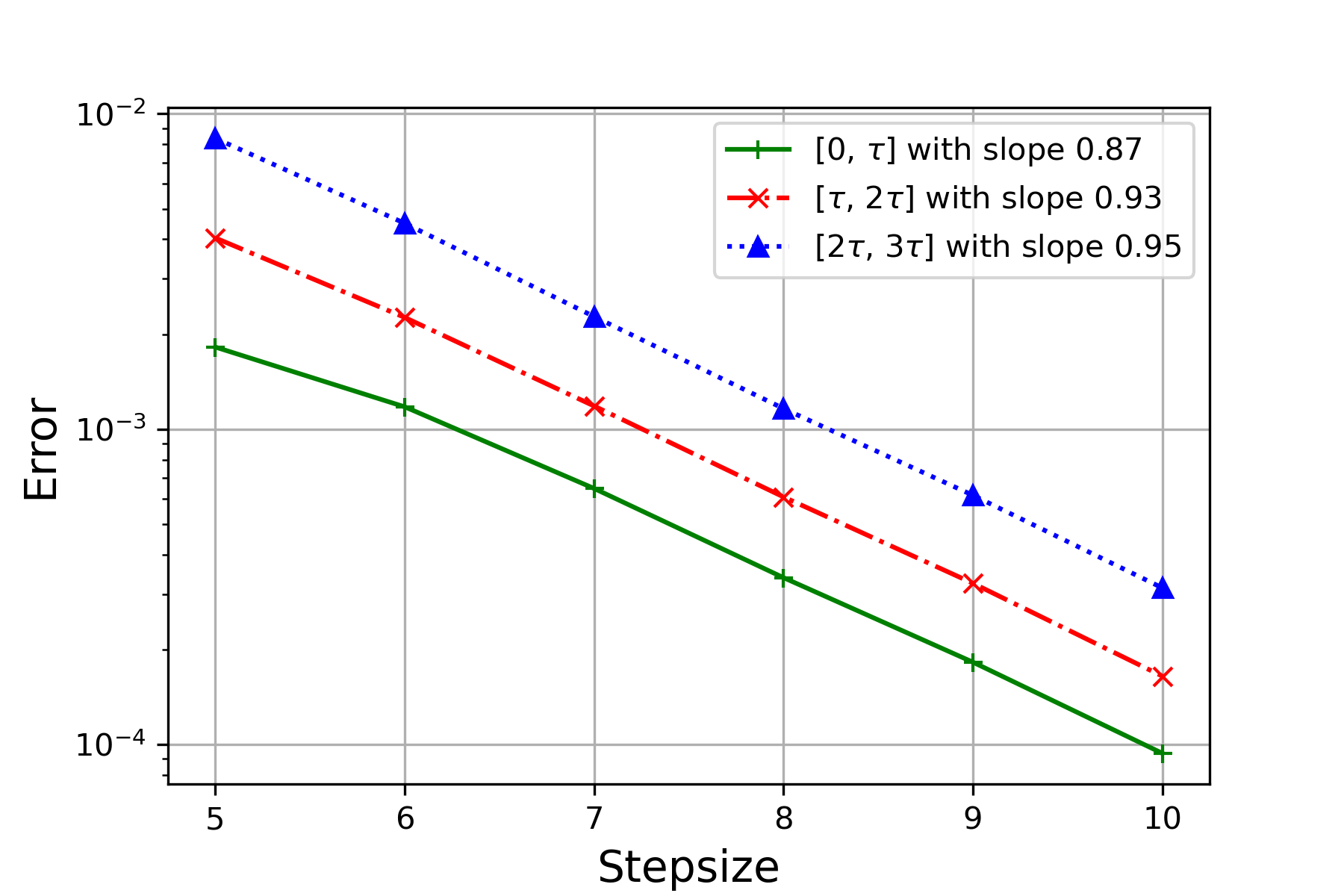}
    \caption{$(\alpha,\gamma)=(0.5,0.1)$.}
    \label{fig:alpha0.5,c0.1,new}
\end{subfigure}
\begin{subfigure}{0.31\textwidth}
    \centering
    \includegraphics[width=\textwidth]{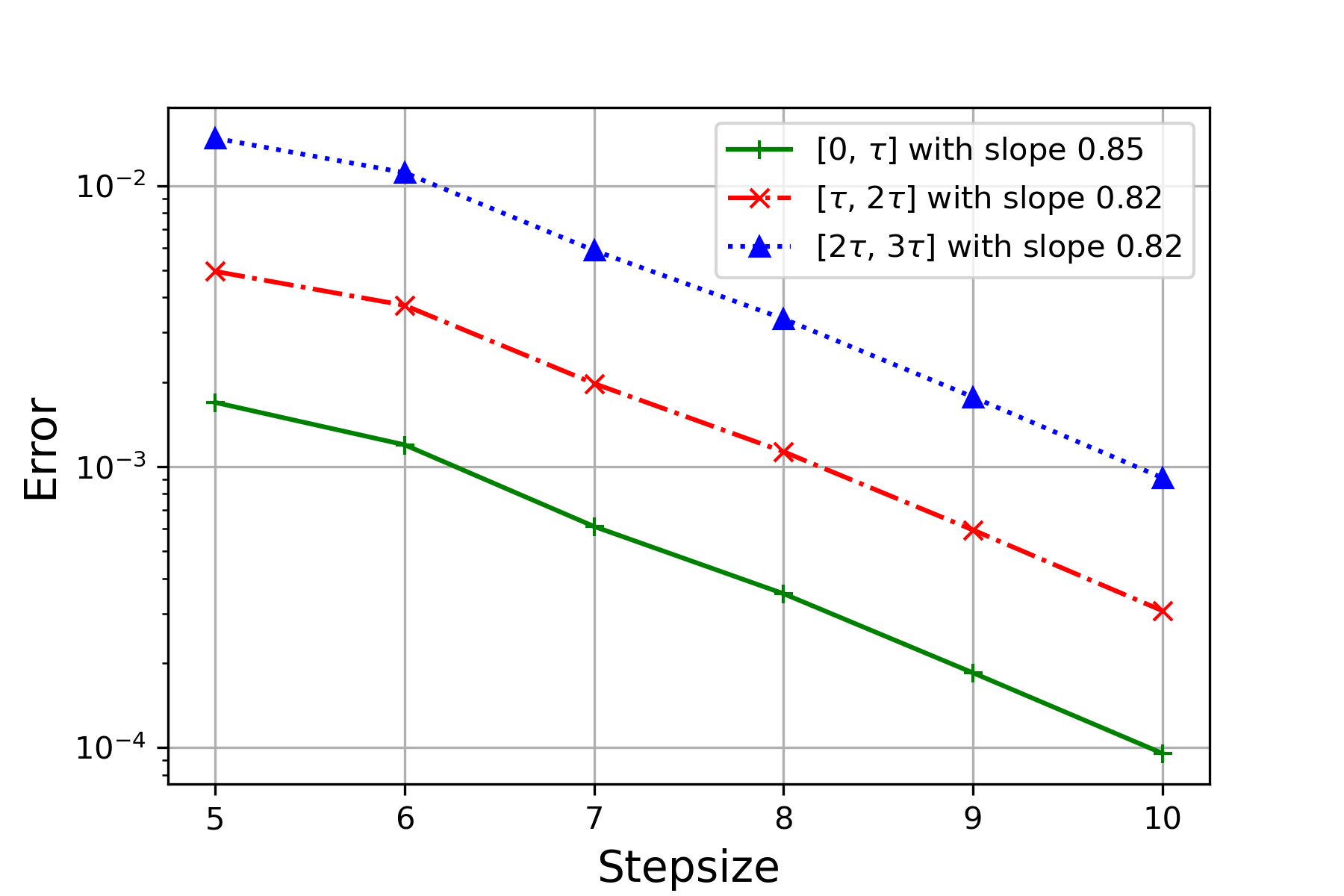}
    \caption{$(\alpha,\gamma)=(0.1,0.5)$.}
    \label{fig:alpha0.1,c0.5,new}
\end{subfigure}\\
\begin{subfigure}{0.31\textwidth}
    \centering
    \includegraphics[width=\textwidth]{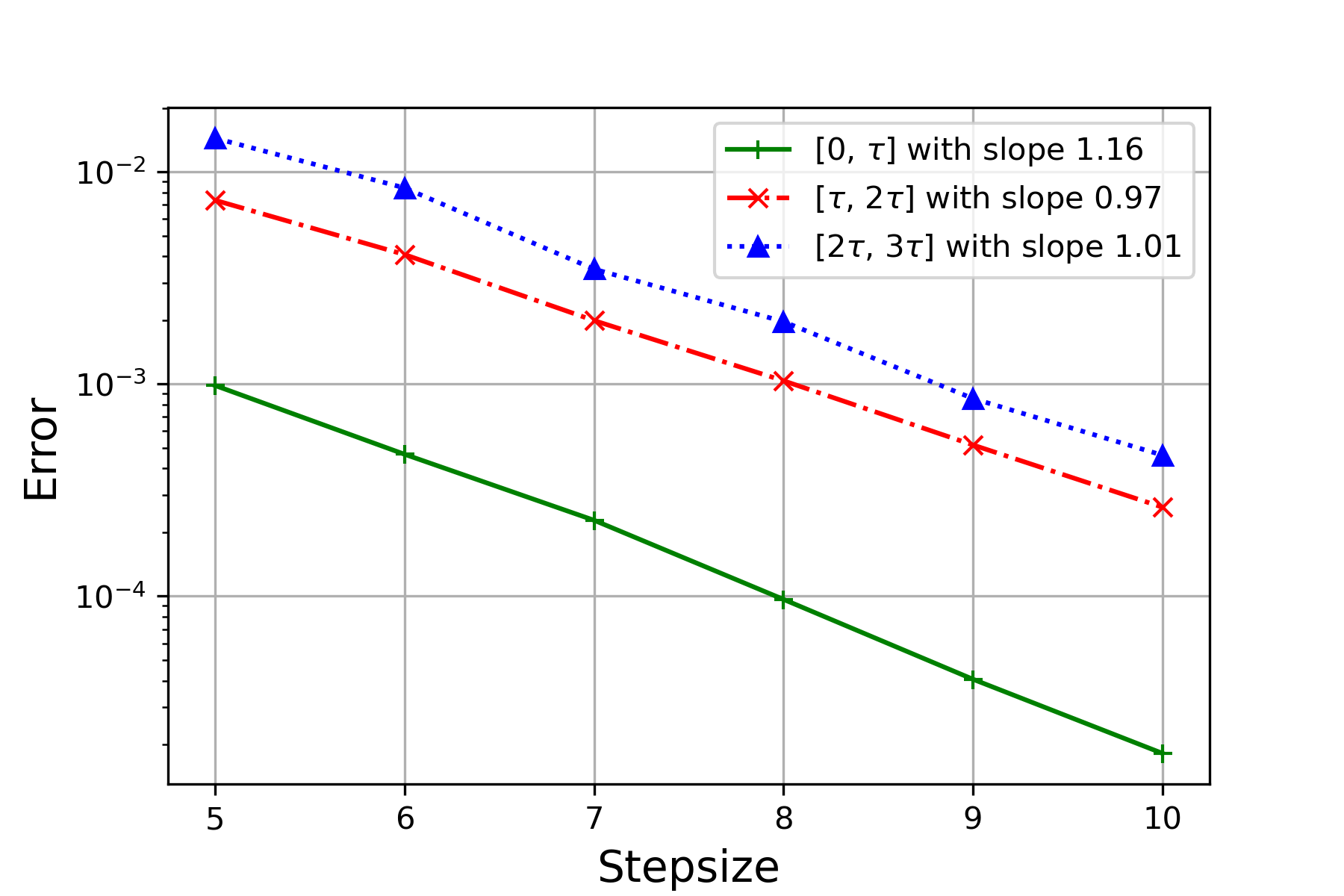}
    \caption{$(\alpha,\gamma)=(0.5,0.5)$.}
    \label{fig:alpha0.5,c0.5,new}
\end{subfigure}
\begin{subfigure}{0.31\textwidth}
    \centering
    \includegraphics[width=\textwidth]{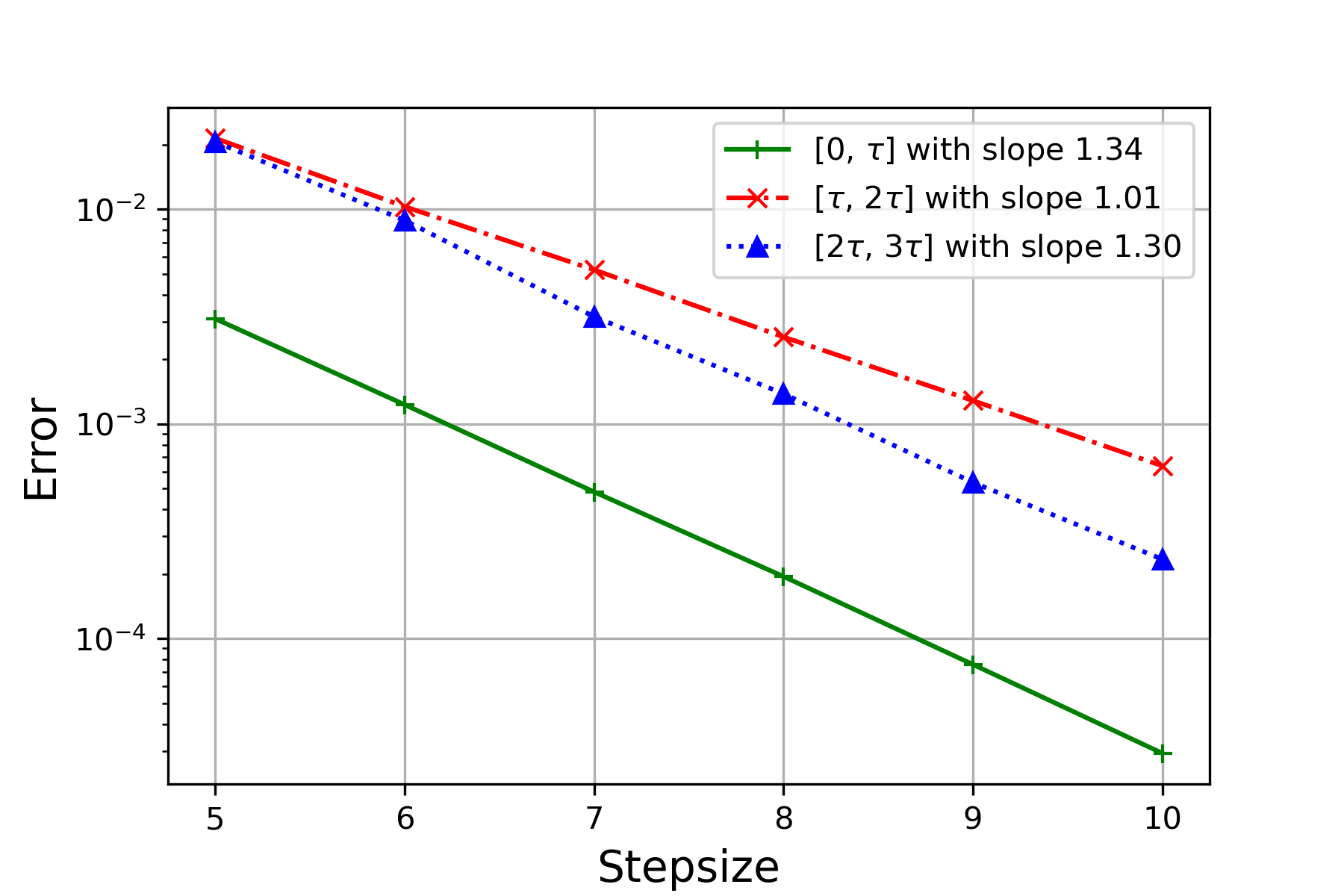}
    \caption{$(\alpha,\gamma)=(0.5,1)$.}
    \label{fig:alpha0.5,c1,new}
\end{subfigure}
\begin{subfigure}{0.31\textwidth}
    \centering
    \includegraphics[width=\textwidth]{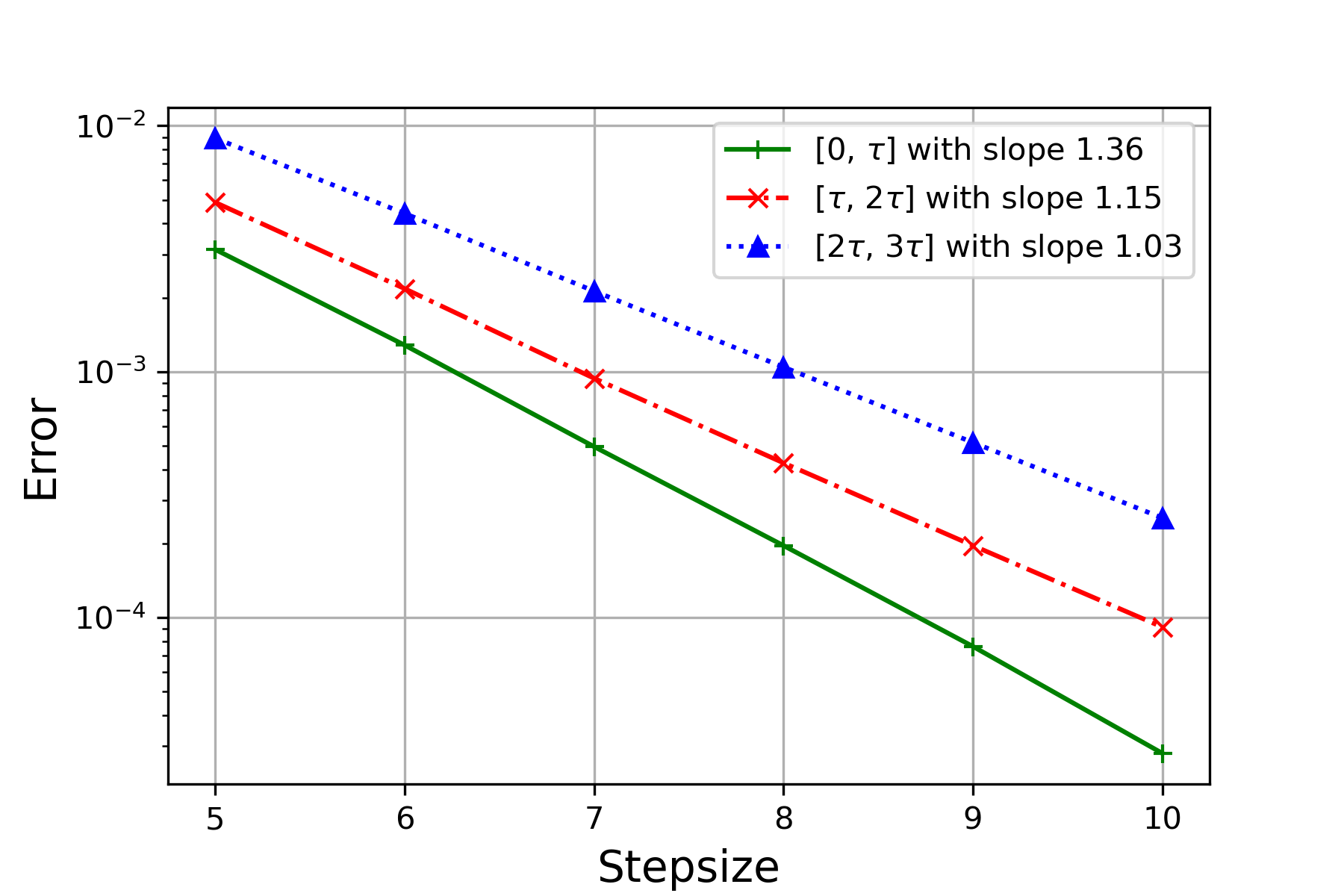}
    \caption{$(\alpha,\gamma)=(1,0.5)$.}
    \label{fig:alpha1,c0.5,new}
\end{subfigure}
\caption{The performance of solving DDE \eqref{eqn:DDEexample2} via randomized Runge-Kutta.\label{fig:newexample}}
\end{figure}

\begin{table}
\centering
\caption{Performance of solving DDE \eqref{eqn:DDEexample2} using randomized Runge-Kutta: negative MSE slopes for varying $\alpha$ and $\gamma$.}
\label{tab:DDEexample2}
\begin{tabular}{cccccccc}
\toprule
$\alpha$ & $\gamma$ & $[0,\tau]$ & $[1\tau,2\tau]$ & $[2\tau,3\tau]$ & Figure \\
\midrule
0.1 & 0.1 & 0.86 & 0.83 & 0.84 & Figure \ref{fig:alpha0.1,c0.1,new} \\
0.5 & 0.1 & 0.87 & 0.93 & 0.95 & Figure \ref{fig:alpha0.5,c0.1,new} \\
0.1 & 0.5 & 0.85 & 0.82 & 0.82 & Figure \ref{fig:alpha0.1,c0.5,new} \\
0.5 & 0.5 & 1.16 & 0.97 & 1.01 & Figure \ref{fig:alpha0.5,c0.5,new} \\
0.5 & 1 & 1.34 & 1.01 & 1.30 & Figure \ref{fig:alpha0.5,c1,new} \\
1 & 0.5 & 1.36 & 1.15 & 1.03 & Figure \ref{fig:alpha1,c0.5,new} \\
\bottomrule
\end{tabular}
\end{table}

\section{Appendix}\label{app}

For the proof of the following continuous version of Gronwall's lemma \cite[pag. 22]{PLBR} see, for example,  \cite[pag. 580]{PARRAS}.
\begin{lem}
\label{gron_cont}
    Let $f:[t_0,T]\to [0,\infty)$ be an integrable function, $\alpha,\beta:[t_0,T]\to\mathbb{R}$ be continuous on $[t_0,T]$, and let $\beta$ be non-decreasing. If for all $t\in [t_0,T]$
    \begin{equation}
        \alpha(t)\leq \beta(t)+\int\limits_{t_0}^t f(s)\alpha(s)ds,
    \end{equation}
    then
    \begin{equation}
        \alpha(t)\leq \beta(t)\exp\Bigl(\int\limits_{t_0}^t f(s)ds\Bigr).
    \end{equation}
\end{lem}
We also use the following weighted discrete version of Gronwall's inequality, see Lemma 2.1 in \cite{RKYW2017}.
\begin{lem}\label{gron_disc}  Consider two nonnegative sequences
$(u_n)_{n\in\mathbb{N}}$ and $(w_n)_{n\in\mathbb{N}}$ which for some given $a\in [0,\infty)$ satisfy
\begin{equation}
    u_n\leq a+\sum\limits_{j=1}^{n-1}w_ju_j, \quad \hbox{for all} \ n\in\mathbb{N},
\end{equation}
then for all $n\in\mathbb{N}$ it also holds true that
\begin{equation}
    u_n\leq a\exp\Bigl(\sum\limits_{j=1}^{n-1}w_j\Bigr).
\end{equation}
\end{lem}

We use the following result concerning properties of solutions of Carath\'eodory ODEs. It follows from
\cite[Theorem 2.12]{andresgorniewicz1}  (compare also with \cite[Proposition 4.2]{RKYW2017}).
\begin{lem}
	\label{lem_ode_1}
	Let us consider the following ODE
	\begin{equation}
	\label{ODE_1_Peano}
	    z'(t)=g(t,z(t)), \quad t\in [a,b], \quad z(a)=\xi,
	\end{equation}
	where $-\infty<a<b<+\infty$, $\xi\in\R^d$ and $g:[a,b]\times\R^d\to\R^d$ satisfies the following conditions
	\begin{enumerate}[label=\textbf{(G\arabic*)},ref=(G\arabic*)]
		\item\label{ass:G1}
		for all $t\in [a,b]$ the function $g(t,\cdot):\R^d\to\R^d$ is continuous,
		\item\label{ass:G2} 
		for all $y\in\R^d$ the function $g(\cdot,y):[a,b]\to\R^d$ is Borel measurable,
		\item\label{ass:G3} 
		{there exists $K\in (0,\infty)$ such that for all $t\in [a,b]$, $y\in {\mathbb{R}^d}$}
		\begin{displaymath}
		{|g(t,y)|\leq K(1+|y|),}
		\end{displaymath}
		\item\label{ass:G4} 
    {there exists $L\in (0,\infty)$ such that for all $t\in [a,b]$, $x,y\in {\mathbb{R}^d}$}
		\begin{displaymath}
       { |g(t,x)-g(t,y)|\leq L |x-y|.}
		\end{displaymath}

	\end{enumerate}	
	Then \eqref{ODE_1_Peano} has a unique absolutely continuous solution $z:[a,b]\to\R^d$ such that
	\begin{equation}
	\label{est_sol_z}
	{\sup\limits_{t\in[a,b]}|z(t)|\leq (|\xi|+K(b-a))e^{K{(b - a)}}.
	}\end{equation}
	{Moreover, for all $t,s\in [a,b]$}
 {
	\begin{equation}
	\label{lip_sol_z}
	|z(t)-z(s)|\leq \bar K |t-s|,
	\end{equation} }
	{where $\bar K=K\cdot\Bigl(1+(|\xi|+K(b-a))e^{K(b - a)}\Bigr)$}.
\end{lem}
{
\begin{proof}
    The proof for our proposition follows the approach outlined in Lemma 7.1 of \cite{difonzo2024existence}. This is feasible because our assumptions (G3) and (G4) are stronger than \cite[Assumptions (G3) and (G4)]{difonzo2024existence}. Particularly, \cite[Eqn. (7.8) and Eqn. (7.9)]{difonzo2024existence} will undergo modifications as follows due to the strengthened conditions:
\begin{align}
\begin{split}
  |z(t)-z(s)| 
\leq K(1+\sup\limits_{a\leq t\leq b}|z(t)|)|t-s|.  
\end{split}
\end{align}
Consequently, this gives rise to our refined conclusions.
\end{proof}
}

\bibliographystyle{plain}
\bibliography{sample}

\begin{thebibliography}{10}

\bibitem{andresgorniewicz1}
J.~Andres and L.~G\'orniewicz.
\newblock {\em Topological Fixed Point Principles for Boundary Value Problems, vol. I}.
\newblock Springer Science+Business Media Dordrecht, 2003.

\bibitem{bellen1}
A.~Bellen and M.~Zennaro.
\newblock {\em Numerical methods for delay differential equations}.
\newblock Oxford, New York, 2003.

\bibitem{BGMP2021}
T.~Bochacik, M.~Go\'cwin, P.~M. Morkisz, and P.~Przyby{\l}owicz.
\newblock Randomized {R}unge-{K}utta method--{S}tability and convergence under inexact information.
\newblock {\em J. Complex.}, 65:101554, 2021.

\bibitem{bochacik2}
T.~Bochacik and P.~Przyby\l owicz.
\newblock On the randomized {E}uler schemes for {ODEs} under inexact information.
\newblock {\em Numer. Algor.}, 91:1205--1229, 2022.

\bibitem{Brunner2012}
H.~Brunner and E.~Hairer.
\newblock {A general-purpose implicit-explicit Runge-Kutta integrator for delay differential equations}.
\newblock {\em BIT Numerical Mathematics}, 52(2):293--314, 2012.

\bibitem{Eng_DDEs_1}
N.~Czy\.z{}ewska, P.~Morkisz, J.~Kusiak, P.~Oprocha, M.~Pietrzyk, P.~Przyby{\l}owicz, {\L}.~Rauch, and D.~Szeliga.
\newblock On mathematical aspects of evolution of dislocation density in metallic materials.
\newblock {\em IEEE Access}, 10:86793--86811, 2022.

\bibitem{CZPMPP}
N.~Czy\.z{}ewska, P.~M. Morkisz, and P.~Przyby{\l}owicz.
\newblock Approximation of solutions of {DDEs} under nonstandard assumptions via {Euler} scheme.
\newblock {\em Numer. Algor.}, 91:1829--1854, 2022.

\bibitem{daun1}
T.~Daun.
\newblock On the randomized solution of initial value problems.
\newblock {\em J. Complex.}, 27:300--311, 2011.

\bibitem{difonzo2024existence}
Fabio~V. Difonzo, Pawe{\l} Przyby{\l}owicz, and Yue Wu.
\newblock {Existence, uniqueness and approximation of solutions to Carath{\'e}odory delay differential equations}.
\newblock {\em Journal of Computational and Applied Mathematics}, 436:115411, 2024.

\bibitem{Eulalia2001}
C.~A. Eulalia and C.~Lubich.
\newblock {An analysis of multistep Runge-Kutta methods for delay differential equations}.
\newblock {\em Mathematical and Computer Modelling}, 34(10-11):1197--1213, 2001.

\bibitem{hale1977theory}
J.~K. Hale.
\newblock {\em Theory of Functional Differential Equations}.
\newblock Applied Mathematical Sciences. Springer New York, 1977.

\bibitem{Hale1993IntroductionTF}
J.~K. Hale and S.~M.~V. Lunel.
\newblock {\em Introduction to Functional Differential Equations}.
\newblock Springer-Verlag, New York, 1993.

\bibitem{hein_milla1}
S.~Heinrich and B.~Milla.
\newblock The randomized complexity of initial value problems.
\newblock {\em J. Complex.}, 24:77--88, 2008.

\bibitem{jen_neuen1}
A.~Jentzen and A.~Neuenkirch.
\newblock A random {E}uler scheme for {C}arath\'eodory differential equations.
\newblock {\em J. Comp. Appl. Math.}, 224:346--359, 2009.

\bibitem{BK2006}
B.~Kacewicz.
\newblock Almost optimal solution of initial-value problems by randomized and quantum algorithms.
\newblock {\em J. Complex.}, 22:676--690, 2006.

\bibitem{RKYW2017}
R.~Kruse and Y.~Wu.
\newblock Error analysis of randomized {R}unge--{K}utta methods for differential equations with time-irregular coefficients.
\newblock {\em Comput. Methods Appl. Math.}, 17:479--498, 2017.

\bibitem{kruse2017randomized}
Raphael Kruse and Yue Wu.
\newblock A randomized milstein method for stochastic differential equations with non-differentiable drift coefficients.
\newblock {\em arXiv preprint arXiv:1706.09964}, 2017.

\bibitem{Kuehn2004}
J.~Kuehn.
\newblock {\em Numerical methods for delay differential equations}.
\newblock World Scientific Publishing Company, 2004.

\bibitem{PARRAS}
E.~Pardoux and A.~Rascanu.
\newblock {\em Stochastic Differential Equations, Backward SDEs, Partial Differential Equations}.
\newblock Stochastic Modelling and Applied Probability. Springer International Publishing Switzerland, 2014.

\bibitem{PLBR}
E.~Platen and N.~Bruti-Liberati.
\newblock {\em Numerical Solution of Stochastic Differential Equations with Jumps in Finance}.
\newblock Stochastic Modelling and Applied Probability. Springer--Verlag Berlin Heidelberg, 2010.

\bibitem{Song2015}
Y.~Song and X.~Yang.
\newblock {A novel explicit two-stage Runge-Kutta method for delay differential equations with constant delay}.
\newblock {\em Applied Mathematics and Computation}, 272:317--322, 2015.

\end{thebibliography}
\end{document}